\documentclass[12pt]{amsart}
\usepackage{amssymb}
\usepackage[all]{xy}
\usepackage{tabulary,caption}
\theoremstyle{plain}

\newtheorem{Thm}{Theorem}[section]
\newtheorem{Cor}[Thm]{Corollary}
\newtheorem{Lem}[Thm]{Lemma}
\newtheorem{Prop}[Thm]{Proposition}
\newtheorem{Conj}[Thm]{Conjecture}
\newtheorem{OP}[Thm]{Open Problem}
\newtheorem*{thma}{Main Theorem}

\theoremstyle{definition}
\newtheorem{Def}[Thm]{Definition}

\theoremstyle{remark}
\newtheorem{Rem}[Thm]{Remark}

\makeatletter
\@namedef{subjclassname@2020}{%
  \textup{2020} Mathematics Subject Classification}
\makeatother

\numberwithin{equation}{subsection}

\begin{document}

\title[Holomorphic Factorization of Mappings into $ \operatorname{Sp}_{4}( \mathbb{C}) $]%
{Holomorphic Factorization of Mappings into $ \operatorname{Sp}_{4}( \mathbb{C}) $}
\author{Bj\"orn Ivarsson \and Frank Kutzschebauch \and Erik L{\o}w}
\address{Department of Mathematics of Systems Analysis\\
Aalto University\\
P.O. Box 11100, FI--00076 Aalto, Finland}
\address{Departement Mathematik\\
Universit\"at Bern\\
Sidlerstrasse 5, CH--3012 Bern, Switzerland}
\address{Department of Mathematics\\
University of Oslo\\
P.O. Box 1053, Blindern, NO--0316 Oslo, Norway}
\email{bjorn.ivarsson@aalto.fi}
\email{frank.kutzschebauch@math.unibe.ch}
\email{elow@math.uio.no}
\thanks{Part of this research was done while the authors were visitors at The Centre for Advanced 
Study (CAS) at the Norwegian Academy of Science and Letters. Bj\"orn Ivarsson was also 
supported by the Magnus Ehrnrooth Foundation and Erik L\o w by Bergens Forskningsstiftelse (BFS). The research of Frank Kutzschebauch  was partially supported by Schweizerische Nationalfonds Grant 200021-178730.}

\date{15 May, 2020}
\setcounter{tocdepth}{3}
†\begin{abstract}
We prove that any null-homotopic holomorphic map from a Stein space $X$ to the symplectic group $\operatorname{Sp}_{4}(\mathbb{C})$ can be written as a finite product of elementary symplectic matrices with
holomorphic entries.
\end{abstract}
\maketitle
\bibliographystyle{amsalpha}
\tableofcontents
\section{Introduction}\label{introduction}

The continuous or holomorphic parameter dependence of classical linear algebra results over the fields $\mathbb{R}$ or $\mathbb{C}$ form a circle  of very natural  questions of general mathematical interest. For example the factorization of continuous  matrices as a product of  continuous elementary matrices has been studied and solved  by Vaserstein \cite{Vaserstein}. The corresponding holomorphic problem for the special linear group $\operatorname{SL}_n$ has been posed by Gromov \cite{Gromov:1989} and finally be solved by  the first two authors in \cite{Ivarsson:2012}. The study of algebraic dependence is connected with famous work by Suslin \cite{Suslin}, Cohn \cite{Cohn}, Bass, Milnor, Serre \cite{BMS} and many others.

These parameter dependence questions  are  a part of algebraic $K$-theory and the study of linear algebra over general rings.  Factorization of Chevalley groups over $\mathbb{R}$ and $ \mathbb{C} $ into elementary matrices is classically well known. For Chevalley groups over general rings this is much more difficult and studied a lot. For an overview see for example the paper by Vavilov and Stepanov \cite{VS}.

 Our main interest are the rings of holomorphic functions on Stein spaces. The only known holomorphic result is the existence for the special linear groups in \cite{Ivarsson:2012}, where Gromov's problem is solved in full generality. In the special case of an open Riemann surface the problem was solved earlier (absolutely unnoticed) by Klein and Ramspott in \cite{Klein:1988}. As well the authors proved the main result 
 of this paper for any size of symplectic matrices in the special case of of an open Riemann surface in \cite{Ivarsson:2019}.

 In the present paper we consider the symplectic groups over rings of holomorphic function on Stein spaces. The main result is (see Section \ref{s:continuous} for notation)

\begin{thma}[also Theorem \ref{t:mainthmrestate}]\label{t:mainthm}
Let $X$ be a finite dimensional reduced Stein space and $f\colon X\to \operatorname{Sp}_{4}(\mathbb{C})$ be a holomorphic mapping that is null-homotopic. Then there exist a natural number $K$ and holomorphic mappings \[G_1,\dots, G_{K}\colon X\to \mathbb{C}^{3}\] such that \[f(x)=M_{1}(G_1(x))\dots M_{K}(G_{K}(x)).\]
\end{thma}
We remind the reader that a mapping is null-homotopic if it is homotopic to a constant map. By Grauert's Oka principle
it is equivalent for a holomorphic map from a Stein space into a complex Lie group to be null-homotopic via holomorphic maps or via continuous maps.

 Our main tool is  the Oka principle for stratified elliptic submersions, the most elaborate result in modern Oka theory. In order to apply an Oka principle one needs a topological solution which we take from our previous work on symplectic groups over rings of continuous functions on topological spaces. The Oka principle lets us homotope the topological solution to a holomorphic one. The technical details needed to prove that certain fibrations are stratified elliptic are considerable and we have so far only been able to complete these details for $\operatorname{Sp}_4$. We expect that a similar result holds for $\operatorname{Sp}_{2n}$.
 
 Factorization of symplectic groups over other rings (of mainly algebraic nature)  has been considered before for example by Kopeiko \cite{Kopeiko}, Grunewald, Mennicke and Vaserstein in \cite{Grunewald:1991}. We have especially included a section, Section \ref{K-theory}, where we explain how our results can be formulated in algebraic/K-theoretic terms.

The paper is organized as follows. In Section \ref{s:continuous} we recall our results on factorization of continuous matrices and prove a slight extension
about the number of factors. In Section \ref{overview} we state our main results and give an overview over the proof. In Section \ref{K-theory} we explain
how our results can be reformulated in the language used in algebraic $K$-theory. In Section \ref{sprays} we recall the theorems from Oka theory which we use in our proof.

In Section \ref{technical1} we give the proofs of Lemmata \ref{l:surjective} and \ref{l:submersive} where we prove that
the most important  fibrations in this paper, the projections of products of elementary symplectic matrices onto their last row, are surjective and we
determine where they are submersive. This is done for symplectic matrices of all sizes, since we hope to be able in the future to
prove that these fibrations are stratified elliptic for all sizes.

The rest of the paper is devoted to prove that our
fibration (for $(4\times4)$-matrices) is stratified elliptic in order to 
be able to apply Oka theory.
In Section \ref{s:stratification} we describe the 
stratification with respect to which we want to prove that the important fibration is stratified
elliptic. This has to do with how the set of $2n$ algebraic equations defining a fiber in the fibration can be reduced to $n$ equations. 
In the case of the Special Linear Group in \cite{Ivarsson:2012} we were able to reduce to one 
single equation independent of the size of the matrices, which was the crucial trick to prove 
ellipticity by using Gromov's example of a spray, complete vector fields. This inability to reduce to less equations is the main difference between  present situation of the Symplectic Group and the Special Linear Group. It  causes all the  hard technical work which fills the rest of the paper. In the next Section \ref{s:descComplete} we introduce our method to find complete vector fields tangent to the fibration. However not all of them are complete and we deduce that  the Gromov-spray produced by them is not dominating.  We determine which of them are complete. In Section \ref{s:strategy}
we explain our strategy to enlarge the set of complete vector fields so that this enlarged 
collection now spans the tangent space at all points and thus gives a fibre dominating spray. The realization of this strategy takes Sections \ref{s:helpful}, where we introduce useful quantities, Sections  \ref{s:proof3factors}, \ref{s:proof4factors}, \ref{s:proof5factors}, where we prove the result for $3, 4$ and $5$  (elementary symplectic) factors, and finally we can give an inductive (over the number of factors) proof in Section \ref{s:induction}. The reason for dealing with the low numbers of factors separately is that some
of the fibers of our fibration are reducible in the cases of small numbers of factors, and from $5$ factors on all fibers are irreducible.
In the last Section \ref{s:exponential} we end the paper with an application to the problem of product of exponentials and formulate some open questions.

\section{Continuous factorization}\label{s:continuous}

Let $ \omega = \sum_{j=1}^n dz_j\wedge dz_{j+n} $ be the symplectic form in $ \mathbb{C}^{2n} $. With respect to $\omega$ symplectic matrices are those that can be written in block form as \[ \begin{pmatrix} A & B \\ C & D \end{pmatrix} \] where $ A, B, C $ and $ D $ are complex $ (n \times n) $ matrices satisfying 
\begin{equation}\label{e:firstsymp} 
A^TC=C^TA
\end{equation}
\begin{equation}\label{e:secondsymp}
B^TD=D^TB
\end{equation}
\begin{equation}\label{e:thirdsymp}
A^TD-C^TB=I_n 
\end{equation}
where  $ I_n $ is the  $ (n\times n) $ identity matrix. In the special case $B=C=0$ this means that $D=(A^T)^{-1}$ and in the special case $A=D=I_n$ this means that $B$ and $C$ are symmetric and $C^T B=0$.  Let $ U_n $ denote a $ (n\times n) $-matrix satisfying $ U_n=U_n^T $ and $ 0_n $ the  $ (n\times n) $ zero matrix.  We call those matrices that are written in block form as \[ \begin{pmatrix} I_n & 0_n \\ U_n & I_n  \end{pmatrix} \mbox{ or }  \begin{pmatrix} I_n & U_n \\ 0_n & I_n  \end{pmatrix} \] {\it elementary symplectic matrices}. Let \[ U_n(x_1,\dots, x_{n(n+1)/2}) = \begin{pmatrix} x_1 & x_2 & \dots & x_n \\ x_2 & x_{n+1} & \dots & x_{2n-1} \\ \vdots & \vdots & \ddots & \vdots \\ x_n & x_{2n-1} & \dots & x_{n(n+1)/2} \end{pmatrix}. \] Given a map $ G\colon X \to \mathbb{C}^{n(n+1)/2} $ let \[ U_n(G(x))=U_n(G_1(x),\dots, G_{n(n+1)/2}(x)) \] where the $ G_j $'s are components of the map $ G $. For odd $ k $ let \[ M_k(G(x))= \begin{pmatrix} I_n & 0_n \\ U_n(G(x)) & I_n \end{pmatrix} \] and for even $ k $ \[ M_k(G(x))= \begin{pmatrix} I_n & U_n(G(x)) \\ 0_n & I_n \end{pmatrix}. \] 

The following result is a refinement of (\cite[Theorem 1.3]{Ivarsson:2019}).

\begin{Thm}\label{t:mainthmtoprestate}
(Continuous Vaserstein problem for symplectic matrices)
There exists a natural number $K(n,d)$ such that given any  finite dimensional normal topological space $X$ of (covering) dimension $d$ and any null-homotopic continuous mapping $M\colon X\to \operatorname{Sp}_{2n}(\mathbb{C})$ there exist $K$ continuous mappings \[G_1,\dots, G_{K}\colon X\to \mathbb{C}^{n(n+1)/2}\] such that \[M(x)=M_{1}(G_1(x))\dots M_{K}(G_{K}(x)).\]
\end{Thm}

\begin{proof}
Theorem 1.3 in \cite{Ivarsson:2019} does not give a uniform bound on the number of factors depending on $n$ and $d$.
Suppose such a bound  would not exist, i.e., for all natural numbers $i$ there are normal topological  spaces $X_i$ of dimension $d$ and null-homotopic continuous  maps $f_i \colon X_i\to \operatorname{Sp}_{2n}(\mathbb{C})$ such that $f_i$ does not factor over a product of less than $i$ elementary symplectic  matrices. Set $X = \cup_{i=1}^\infty X_i$ the disjoint union of the spaces $X_i$ and
$F\colon X\to \operatorname{Sp}_{2n}(\mathbb{C})$ the map that is equal to $f_i$ on $X_i$. By Theorem 1.3. in \cite{Ivarsson:2019}
$F$ factors over a finite number of elementary symplectic matrices. Consequently all $f_i$ factor over the same number of elementary symplectic matrices which contradicts the assumption on $f_i$.

\end{proof}

\section{Statement of the main result and overview of proof}
\label{overview}

We state the main result of this paper which is a holomorphic version of Theorem \ref{t:mainthmtoprestate} for $\operatorname{Sp}_4(\mathbb{C})$.

\begin{Thm}\label{t:mainthmrestate}
There exists a natural number $N(d)$ such that given any finite dimensional reduced Stein space $X$ of dimension $d$ and any null-homotopic holomorphic mapping  $f\colon X\to \operatorname{Sp}_{4}(\mathbb{C})$  there exist  $N$ holomorphic mappings \[G_1,\dots, G_{N}\colon X\to \mathbb{C}^{3}\] such that \[f(x)=M_{1}(G_1(x))\dots M_{N}(G_{N}(x)).\]
\end{Thm}

We have the following corollary.

\begin{Cor}
Let $X$ be a finite dimensional reduced Stein space that is topologically contractible and $f\colon X\to\operatorname{Sp}_{4}(\mathbb{C})$ be a holomorphic mapping. Then there exist a natural number $N$ and holomorphic mappings \[G_1,\dots, G_{N}\colon X\to \mathbb{C}^{3}\] such that \[f(x)=M_{1}(G_1(x))\dots M_{N}(G_{N}(x)).\]
\end{Cor}



The strategy for proving Theorem \ref{t:mainthmrestate} is as follows. Define \[\Psi_K\colon (\mathbb{C}^{3})^K\to \mbox{Sp}_4(\mathbb{C})\] as 
\begin{equation} \label{e:product}
\Psi_K(x_1,\dots,x_{3K})=M_1(x_1,x_2,x_3)\dots M_K(x_{3K-2}, x_{3K-1}, x_{3K}).
\end{equation}  We want to show the existence of a holomorphic map \[G=(G_1,\dots, G_K)\colon X\to (\mathbb{C}^{3})^K\] such that  \[\xymatrix{ & (\mathbb{C}^{3})^K \ar[d]^{\Psi_{K}} \\ X \ar[r]_{f} \ar[ur]^{G} & \mbox{Sp}_4(\mathbb{C})}\] is commutative. Theorem \ref{t:mainthmtoprestate} shows the existence of a continuous map such that the diagram above is commutative. 

We will prove Theorem \ref{t:mainthmrestate} using the \textsc{Oka-Grauert-Gromov} principle for sections of holomorphic submersions over $X$. One candidate submersion would be to use the pull-back of $\Psi_K\colon (\mathbb{C}^{3})^K\to \mbox{Sp}_4(\mathbb{C})$. It turns out that $\Psi_K$ is not a submersion at all points in $(\mathbb{C}^{3})^K$. It is a surjective holomorphic submersion if one removes a certain subset from $(\mathbb{C}^{3})^K$. Unfortunately the fibers of this submersion are quite difficult to analyze and we therefore elect to study  \[\xymatrix{ & (\mathbb{C}^{3})^K \ar[d]^{\pi_4\circ \Psi_{K}} \\ X \ar[r]_{\pi_4\circ f} \ar[ur]^{F} & \mathbb{C}^4\setminus \{ 0 \}}\] where we define the projection $\pi_4\colon \mbox{Sp}_4(\mathbb{C})\to \mathbb{C}^4\setminus \{ 0 \}$ to be the
projection of a matrix to its last row:
\begin{equation*}
\pi_4 \begin{pmatrix} z_{11} & \dots & z_{14} \\ \vdots & \ddots & \vdots \\ z_{41} & \dots & z_{44} \end{pmatrix} =(z_{41},\dots , z_{44}).
\end{equation*} 

However, even the map $\Phi_K=\pi_4\circ \Psi_K\colon (\mathbb{C}^{3})^K\to \mathbb{C}^4\setminus \{0\}$ is not submersive everywhere. We have the three results below (Lemma \ref{l:surjective}, Lemma \ref{l:submersive} and Proposition \ref{p:mainprop})  about that map which will be proved in later sections. 

We introduce some notation. Projecting to the last row introduces an asymmetry between upper and lower triangulary elementary matrices and therefore we will denote by $ z $'s the variables in the lower triangular matrices and $ w $'s the variables in the upper triangular matrices. For example, the right hand side of (\ref{e:product}) becomes \[ \begin{pmatrix} 1 & 0 & 0 & 0 \\ 0 & 1 & 0 & 0 \\ z_1 & z_2 & 1 & 0 \\ z_2 & z_3 & 0 & 1 \end{pmatrix}\begin{pmatrix} 1 & 0 & w_1 & w_2\\ 0 & 1 & w_2 & w_3 \\ 0 & 0 & 1 & 0 \\ 0 & 0 & 0 & 1 \end{pmatrix}\cdots \begin{pmatrix} 1 & 0 & w_{3k-2} & w_{3k-1}\\ 0 & 1 & w_{3k-1} & w_{3k} \\ 0 & 0 & 1 & 0 \\ 0 & 0 & 0 & 1 \end{pmatrix} \] for even $ K=2k $.

 Let \[ \vec{Z}_K=\begin{cases} (z_1, z_2, z_3, w_1, w_2,w_3,\dots, w_{3k-2},w_{3k-1},w_{3k}) \text{ if } K=2k \\ (z_1, z_2, z_3, w_1, w_2,w_3,\dots, z_{3k+1},z_{3k+2},z_{3k+3}) \text{ if } K=2k+1\end{cases} \] and 
\[ W_K = \begin{cases}  \begin{pmatrix} w_1 & w_2 & w_4 & w_5 & \dots & w_{3k-5} & w_{3k-4} \\ w_2 & w_3 & w_5 & w_6 & \dots & w_{3k-4} & w_{3k-3} \end{pmatrix} \text{ if } K=2k\\ \\ \begin{pmatrix} w_1 & w_2 & w_4 & w_5 & \dots & w_{3k-2} & w_{3k-1} \\ w_2 & w_3 & w_5 & w_6 & \dots & w_{3k-1} & w_{3k} \end{pmatrix}\text{ if } K = 2k+1. \end{cases}  \] 
Also, when $ K=2k $ or $ K=2k+1 $, let \begin{equation*} 
A_K = \bigcap_{1\le j \le k}\left \{\vec{Z}_K\in (\mathbb{C}^{3})^K: z_{3j-1}=z_{3j}=0 \right \}, 
\end{equation*}
\begin{equation*}
B_K = \left \{\vec{Z}_K\in (\mathbb{C}^{3})^K: \operatorname{Rank}W_K < 2 \right \}  
\end{equation*}
and 
\begin{equation}\label{e:submersive} 
S_K = A_K \cap B_K. 
\end{equation}

We have Lemma \ref{l:surjective} that follows from a simple calculation. 

\begin{Lem}\label{l:surjective}
The mapping $\Phi_K=\pi_4\circ \Psi_K\colon (\mathbb{C}^{3})^K\setminus S_K\to \mathbb{C}^4\setminus \{0\}$ is surjective when $K\ge 3$.
\end{Lem}


\begin{Lem}\label{l:submersive}
For $K\ge 3$ the mapping $\Phi_K=\pi_4\circ \Psi_K\colon (\mathbb{C}^3)^K\to \mathbb{C}^4\setminus \{0\}$ is a holomorphic submersion exactly at points $\vec{Z}_K\in (\mathbb{C}^{3})^K\setminus S_K$ where  $ S_K $ is defined by (\ref{e:submersive}) above.
That is, $ S_K $ is the set of points where the entries in the last row of each lower triangular matrix are zero, except for the $K$-th matrix where no conditions are imposed, and the rank of the matrix $ W_K $, which does not involve entries from the $ K $-th matrix, is strictly less than 2.
\end{Lem} 

\begin{Rem} \label{r:generalizeSub}
Lemma \ref{l:surjective} and Lemma \ref{l:submersive} both generalize to $ (2n\times 2n)$-matrices and the proofs are identical. In Section \ref{technical1} we therefore
consider the general case.
\end{Rem}

\begin{Prop} \label{p:mainprop}
For $n=1$ and $n=2$ the map  \begin{equation}\label{e:ellipticsubmersion}\xymatrix{ (\mathbb{C}^{n(n+1)/2})^K \setminus
    S_K\ar[d]^{\pi_{2n}\circ \Psi_{K}} \\  \mathbb{C}^{2n}\setminus \{ 0
    \}}\end{equation} is a stratified elliptic submersion. 
\end{Prop}

\begin{Cor}\label{c:factorization}
Let $n=1$ or $n=2$. Let $X$ be a finite dimensional reduced Stein space and $f\colon X \to
\operatorname{Sp}_{2n}(\mathbb{C})$ be a holomorphic
map. Assume that there exists a natural number $K$ and a continuous map $F\colon X\to (\mathbb{C}^{n(n+1)/2})^K\setminus S_K$ such that  \[\xymatrix{ & (\mathbb{C}^{n(n+1)/2})^K \setminus S_K\ar[d]^{\pi_{2n}\circ \Psi_{K}} \\ X \ar[r]_{\pi_{2n}\circ f} \ar[ur]^{F} & \mathbb{C}^{2n}\setminus \{ 0 \}}\] is commutative. Then there exists a holomorphic map $G\colon X\to (\mathbb{C}^{n(n+1)/2})^K\setminus S_K$, homotopic to $F$ via continuous maps $F_t\colon X\to (\mathbb{C}^{n(n+1)/2})^K\setminus S_K$, such that the diagram above is commutative for all $F_t$.
\end{Cor}


\begin{proof}
The pull back of (\ref{e:ellipticsubmersion}) by $ \pi_{2n} \circ f $ is a stratified elliptic submersion over the Stein base $ X $. Thus by Theorem \ref{t:forstneric} there is a homotopy from the given continuous section to a holomorphic section. This is equivalent to the desired homotopy $ F_t $. An even better way to perform this proof is to say that the map (\ref{e:ellipticsubmersion}) is an Oka map, see \cite[Corollary 6.14.4.(i)]{Forstneric:2011}, which yields the desired conclusion.

\end{proof}

\begin{Rem}\label{continuous parameter} The fact that the map (\ref{e:ellipticsubmersion}) is an Oka map, see \cite[Corollary 6.14.4.(i)]{Forstneric:2011} yields
a parametric version of Corollary \ref{c:factorization}. This means that the holomorphic map can be
replaced by a continuous map $f_P\colon X  \times P \to
\operatorname{Sp}_{2n}(\mathbb{C})$, which is holomorphic for each fixed parameter $p\in P$ and where  $P$ is a compact Hausdorff
topological space.
\end{Rem}

We need the following version of the Whitehead Lemma:

\begin{equation}\label{e:whitehead}.
\begin{aligned}
 & \begin{pmatrix} 1 & 0 & 0 & 0 \\ a & 1 & 0 & 0 \\ 0 & 0 & 1 & -a \\ 0 & 0 & 0 & 1 \end{pmatrix} = \\
 & =\begin{pmatrix} 1 & 0 & 0 & 0 \\ 0 & 1 & 0 & 0 \\ -a & -1 & 1 & 0 \\ -1 & 0 & 0 & 1 \end{pmatrix}
\begin{pmatrix} 1 & 0 & 0 & 0 \\ 0 & 1 & 0 & -a \\ 0 & 0 & 1 & 0 \\ 0 & 0 & 0 & 1 \end{pmatrix}
\begin{pmatrix} 1 & 0 & 0 & 0 \\ 0 & 1 & 0 & 0 \\ 0 & 1 & 1 & 0 \\ 1 & 0 & 0 & 1 \end{pmatrix}
\begin{pmatrix} 1 & 0 & 0 & 0 \\ 0 & 1 & 0 & a \\ 0 & 0 & 1 &  0\\ 0 & 0 & 0 & 1 \end{pmatrix}.\end{aligned}\end{equation}

\begin{proof}[Proof of Theorem \ref{t:mainthmrestate}]
We will prove the theorem for a single map. The existence of a uniform bound $N(d)$ follows as in the proof of Theorem \ref{t:mainthmtoprestate}.
Since a finite dimensional Stein space is finite dimensional as a topological space there are $K-2$ continuous mappings  \[G_1,\dots, G_{K-2}\colon X\to \mathbb{C}^{3}\] such that \[f(x)=M_{1}(G_1(x))\dots M_{K-2}(G_{K-2}(x)).\]
Choose a constant  symmetric $2 \times 2$ matrix $H$ with non-zero second row and replace the above factorization by 
 \[f(x)= M_1 (H) M_2 (0_{2,2}) M_{3}(G_1(x)-H) M_4 (G_2 (x))\dots M_{K}(G_{K-2}(x)).\]
 This factorization by $K$ continuous elementary symplectic matrices avoids the singularity set  $S_K$ and thus we found 
$F\colon X\to (\mathbb{C}^{3})^K\setminus S_K$ with $\Psi_K (F) =f$.

Using Corollary \ref{c:factorization} we know that $F_0:= F$ is homotopic to a holomorphic map $G=F_1$, via continuous maps $F_t$, such that \[\pi_4(f(x))=\pi_4\circ\Psi_K(F_t(x)), \hskip2mm 0 \le t \le 1 ,\] that is the last row of the matrices $\Psi_K(F_t(x))$ is constant. Therefore \begin{equation*}
\Psi_K(F_t(x))f(x)^{-1} = \begin{pmatrix} \widetilde{f}_{11,t}(x) & \widetilde{f}_{12,t}(x) & \widetilde{f}_{13,t}(x) & \widetilde{f}_{14,t}(x)    \\   \widetilde{f}_{21,t}(x) & \widetilde{f}_{22,t}(x) & \widetilde{f}_{23,t}(x) & \widetilde{f}_{24,t}(x) 
\\ \widetilde{f}_{31,t}(x) & \widetilde{f}_{32,t}(x) & \widetilde{f}_{33,t}(x) & \widetilde{f}_{34,t}(x) \\ 0 & 0 & 0 & 1
\end{pmatrix}
\end{equation*}
Since these matrices are symplectic it automatically follows that $ \widetilde{f}_{12,t}(x)\equiv 0 $, $ \widetilde{f}_{22,t}(x)\equiv 1 $ and $ \widetilde{f}_{32,t}(x)\equiv 0 $ so that  \begin{equation} \label{e:lastrow}
\Psi_K(F_t(x))f(x)^{-1} = \begin{pmatrix} \widetilde{f}_{11,t}(x) & 0 & \widetilde{f}_{13,t}(x) & \widetilde{f}_{14,t}(x)    \\   \widetilde{f}_{21,t}(x) & 1 & \widetilde{f}_{23,t}(x) & \widetilde{f}_{24,t}(x) 
\\ \widetilde{f}_{31,t}(x) & 0 & \widetilde{f}_{33,t}(x) & \widetilde{f}_{34,t}(x) \\ 0 & 0 & 0 & 1
\end{pmatrix}
\end{equation} and in addition the matrix \begin{equation} \label{e:submatrix}
 \widetilde{f_t}(x)=\begin{pmatrix} \widetilde{f}_{11,t}(x) & \widetilde{f}_{13,t}(x)     \\  
\\ \widetilde{f}_{31,t}(x) & \widetilde{f}_{33,t}(x) 
\end{pmatrix} \in \operatorname{Sp}_2( \mathbb{C} ) = \operatorname{SL}_2( \mathbb{C} ).
\end{equation} Since $\Psi_K(F_0(x))=f(x)$ we see that $\widetilde{f}_0=Id$ and thus the holomorphic map $\widetilde{f}:=\widetilde{f}_1\colon X \to \operatorname{SL}_2( \mathbb{C}) $ is null-homotopic.  Let $ \psi $ be the standard inclusion of $ \operatorname{Sp}_2 $ in $ \operatorname{Sp}_4 $, see for example \cite{Grunewald:1991}. By the main result from \cite{Ivarsson:2012} the matrix \begin{equation} \label{e:inverse} 
\psi(\widetilde{f}(x)^{-1})=\begin{pmatrix} \widetilde{f}_{33}(x) & 0 & -\widetilde{f}_{13}(x) & 0    \\   0 & 1 & 0 & 0 
\\ -\widetilde{f}_{31}(x) & 0 & \widetilde{f}_{11}(x) & 0 \\ 0 & 0 & 0 & 1
\end{pmatrix}
\end{equation} is a product of holomorphic elementary symplectic matrices. Therefore it suffices to show that 
\begin{equation}\label{e:prodinv} \Psi_K(G(x))f(x)^{-1} \cdot \psi(\widetilde{f}(x)^{-1})=  \begin{pmatrix} 1 & 0 & 0 & \widetilde{f}_{14}(x)    \\   -\widetilde{f}_{34}(x) & 1 & \widetilde{f}_{14}(x) & \widetilde{f}_{24}(x) 
\\ 0 & 0 & 1 & \widetilde{f}_{34}(x) \\ 0 & 0 & 0 & 1
\end{pmatrix} \end{equation} is a product of elementary symplectic matrices. In order to deduce (\ref{e:prodinv}) one has to use the fact that (\ref{e:lastrow}) is symplectic. Since 
\[ \begin{aligned} &\begin{pmatrix} 1 & 0 & 0 & \widetilde{f}_{14}(x)    \\   -\widetilde{f}_{34}(x) & 1 & \widetilde{f}_{14}(x) & \widetilde{f}_{24}(x) 
\\ 0 & 0 & 1 & \widetilde{f}_{34}(x) \\ 0 & 0 & 0 & 1
\end{pmatrix} = \\ &= \begin{pmatrix} 1 & 0 & 0 & 0  \\   -\widetilde{f}_{34}(x) & 1 & 0 & 0 
\\ 0 & 0 & 1 & \widetilde{f}_{34}(x) \\ 0 & 0 & 0 & 1
\end{pmatrix} \begin{pmatrix} 1 & 0 & 0  & \widetilde{f}_{14}(x)    \\   0 & 1 & \widetilde{f}_{14}(x) & \widetilde{f}_{14}(x) \widetilde{f}_{34}(x)+\widetilde{f}_{24}(x)   
\\ 0 & 0 & 1 & 0 \\ 0 & 0 & 0 & 1
\end{pmatrix}\end{aligned}\] the result follows by (\ref{e:whitehead}).
 
 Analysing this proof and using Remark
 \ref{continuous parameter} one sees that
 we can actually prove a parametric version of our main theorem.
 
 \begin{Thm}
Let $X$ be a finite dimensional reduced Stein space, $P$ a compact Hausdorff topological (parameter) space and $f\colon P\times X\to \operatorname{Sp}_{4}(\mathbb{C})$ be a continuous mapping, holomorphic for each fixed $p\in P$,  that is null-homotopic. Then there exist a natural number $K$ and continuous mappings, holomorphic for each fixed parameter $p\in P$, \[G_1,\dots, G_{K}\colon P\times X\to \mathbb{C}^{3}\] such that \[f(p,x)=M_{1}(G_1(p,x))\dots M_{K}(G_{K}(p,x)).\]
\end{Thm} 
\end{proof}

In order to complete the proof of the theorem we need to establish Proposition \ref{p:mainprop}, Lemma \ref{l:submersive}, and \ref{l:surjective}.

\begin{Rem} \label{r:generalization}
Proposition  \ref{p:mainprop} is the crucial ingredient in the proof of Theorem \ref{t:mainthmrestate}. Its proof is by far the most difficult part of the paper. As pointed out in Remark \ref{r:generalizeSub}, Lemma \ref{l:submersive} holds for general $n$.  
Also if Proposition \ref{p:mainprop} holds for some $n$ then Corollary
\ref{c:factorization} also holds for that $n$. Moreover the reduction of the size of the symplectic matrix from $ \operatorname{Sp}_{4} $ to $ \operatorname{Sp}_{2} $ done in the proof of Theorem \ref{t:mainthmrestate} generalizes easily to a reduction from $ \operatorname{Sp}_{2n} $ to $ \operatorname{Sp}_{2n-2} $ if Corollary \ref{c:factorization} holds for $n$ (see for example the proof of Lemma 4.4 in \cite{Grunewald:1991}). Therefore if Proposition \ref{p:mainprop} can be proven for $n=1,\dots,m$ then the following holds true. 
\begin{Conj}
 Let $X$ be a finite dimensional reduced Stein space and $f\colon X\to \operatorname{Sp}_{2m}(\mathbb{C})$ be a holomorphic mapping that is null-homotopic. Then there exist a natural number $K$ and holomorphic mappings \[G_1,\dots, G_{K}\colon X\to \mathbb{C}^{m(m+1)/2}\] such that \[f(x)=M_{1}(G_1(x))\dots M_{K}(G_{K}(x)).\]
\end{Conj}
\end{Rem}
In the case of a $1$-dimensional Stein space, i.e. an open Riemann surface, this Conjecture 
has been established in \cite{Ivarsson:2019}.
The condition of null-homotopy is automatically satisfied in this case, since an open Riemann surface is homotopy equivalent to a $1$-dimensional
CW-complex and the group  $\operatorname{Sp}_{2m}(\mathbb{C})$ is simply connected. The proof  uses the analytic ingredient that the Bass stable rank of $\mathcal{O} (X)$ is $1$ for an open Riemann surface and proceeds then by linear algebra arguments.

\section{Formulation in algebraic terms}\label{K-theory}

We relate our results to algebraic K-theory and reformulate
them in those terms. The following is a standard notion:

\begin{Def} For a commutative ring $R$  the set $U_m (R)$ of unimodular rows of lenght $m$ is defined as
$$\{ (r_1, r_2, \ldots, r_m) \in R^m : r_1, r_2, \ldots r_m \ \mathrm{ generate } \ R \ \mathrm{ as \ an\  ideal }\}$$
\end{Def}

In our main example,   the ring $\mathcal{O} (X) $ of holomorphic functions  on a Stein space $X$, a row $(f_1, f_2, \ldots, f_m) \in \mathcal{O}^m (X)$
is unimodular iff the functions $f_1, f_2, \ldots, f_m$ have no common zeros, a well known application of Cartan's Theorem B.

Since null-homotopy is an important assumption in our studies we
denote  the set of null-homotopic unimodular rows in $U_m (\mathcal{O} (X))$ by $U_m^0 (\mathcal{O} (X))$. This set can be seen
as the path-connected component of the space of holomorphic maps from $X$ to $\mathbb{C}^m \setminus \{0\}$ containing the constant map $(0,0, ,\ldots,0,1) = e_m$. By Grauerts Oka principle  $\mathbb{C}^m \setminus \{0\}= \operatorname{GL}_m  (\mathbb{C})/\operatorname{GL}_{m-1}(\mathbb{C})$ is an Oka manifold, therefore the path-connected components of continuous and holomorphic maps $X \to \mathbb{C}^m \setminus \{0\}$ are in bijection. This says that unimodular rows in $U_m (\mathcal{O} (X))$ are null-homotopic in the holomorphic sense iff they are null-homotopic in the continuous sense.

Algebraic K-theorists consider Chevalley groups over rings, in our example we consider the null-homotopic elements of them.

\begin{Def}
$\operatorname{Sp}^0_{2n} (\mathcal{O} (X))$ denotes the group of  null-homotopic holomorphic maps from a Stein space $X$ to the
symplectic group $\operatorname{Sp}_{2n} (\mathbb{C})$, which in other words is the path-connected component of the group $\operatorname{Sp}_{2n} (\mathcal{O} (X))$ containing the identity.  
      
\end{Def}

Again by Grauert's Oka principle  holomorphic maps $X \to \operatorname{Sp}_{2n} (\mathbb{C})$ are homotopic via holomorphic maps iff they are homotopic via continuous maps.

Clearly the last row of a matrix in $\operatorname{Sp}_{2n} (\mathcal{O} (X))$ is unimodular, i.e., an element of $U_{2n} (\mathcal{O} (X))$.  Whether a unimodular row in  $U_{2n} (\mathcal{O} (X))$ is the last row of a matrix in $\operatorname{Sp}_{2n} (\mathcal{O} (X))$ is by Oka-theory a purely topological problem. Let us illustrate this by an example.

Extending a unimodular row to an invertible matrix 
can be reformulated as follows: Given a trivial line sub-bundle 
of the trivial bundle $X \times \mathbb{C}^n$ of rank $n$ over $X$. Can it be complemented by a trivial bundle? 

This of course is not always the case: The (non-trivial) tangent bundle $T$ of the sphere $S^{2n+1}$  ($n\ge 4$) is the complement of the trivial normal bundle $N$
to the sphere $S^{2n+1}$ in $\mathbb{R}^{2n+2}$. To make this a holomorphic example consider $X$ to be a Grauert tube around $S^{2n+1}$, i.e., a
Stein manifold which has a strong deformation retraction $\rho$ onto
its totally real maximal dimensional submanifold $S^{2n+1}$. The
bundle $T$ is replaced by the complexified tangent bundle to the sphere pulled back onto $X$ by the retraction $\rho$ and equipped with its unique structure of holomorphic vector bundle (which is still not  a trivial bundle).
The pull-back of the complexified trivial bundle $N$ is still a
trivial line sub bundle of $X \times \mathbb{C}^{2n}$. Thus we found an example of a holomorphic row which cannot be completed
to an invertible matrix in $\operatorname{Gl_{2n}} (\mathcal{O} (X))$ and thus not to a matrix in $\operatorname{Sp}_{2n} (\mathcal{O} (X))$ either.

For null-homotopic rows the situation is better. 

\begin{Lem} Every element $U_{2n}^0 (\mathcal{O} (X))$ extends to a null-homotopic matrix $A\in \operatorname{Sp^0_{2n}} (\mathcal{O} (X))$.

\end{Lem}

\begin{proof} 
Let $F= (f_1, \ldots, f_{2n}): X \to \mathbb{C}^{2n} \setminus \{ 0\}$ be a null-homotopic holomorphic map, the homotopy to
the constant map $F_1(x) = e_{2n}$ be denoted by $F_t$, $t\in [0,1]$. 
The map $\pi_{2n}: \operatorname{Sp}_{2n} (\mathbb{C})\to \mathbb{C}^{2n} \setminus \{ 0\}$ is a locally trivial 
holomorphic fiber bundle with typical fibre $F \cong Sp_{2n-2} (\mathbb{C}) \times \mathbb{C}^{4n-1}$ which is an Oka manifold.
Our problem is to find a global section of the pull-back of this fibration by the map $F=F_0$. Since a locally  trivial bundle is a Serre fibration and the constant last row can be extended to a constant (thus null-homotopic) symplectic matrix we find a continuous section of this pull-back bundle over the
whole homotopy. Thus the restriction to $X\times \{ 0\}$ is a null-homotopic continuous symplectic matrix. 
Since the fiber $F$ is Oka we find a homotopy to a holomorphic symplectic matrix, which is still null-homotopic. 

\end{proof}
The notion of elementary symplectic matrices over a ring $R$ is the same as explained in Section \ref{s:continuous}.

Let $ W_n $ denote a $ (n\times n) $-matrix with entries in the ring $R$ satisfying $ W_n=W_n^T $ and $ 0_n $ the  $ (n\times n) $ zero matrix.  We call those matrices that are written in block form as \[ \begin{pmatrix} I_n & 0_n \\ W_n & I_n  \end{pmatrix} \mbox{ or }  \begin{pmatrix} I_n & W_n \\ 0_n & I_n  \end{pmatrix} \] {\it elementary symplectic matrices over $R$}. The group generated by them,
the elementary symplectic group, is denoted by $\operatorname{Ep_{2n}} (R)$. We consider the group $\operatorname{Ep_{2n}} (\mathcal{O} (X))$ which is easily seen to be a subgroup of (multiply the symmetric matrices $W_n$ by a real number $t\in [0, 1]$) $\operatorname{Sp^0_{2n}} (\mathcal{O} (X))$.

The meaning of Corollary \ref{c:factorization} in K-theoretic terms is now the following:
\begin{Prop} Let $n=1$ or $n=2$.
 For a Stein space $X$ the group $\operatorname{Ep_{2n}} (\mathcal{O} (X))$ acts transitively on the set of null-homotopic  unimodular rows $U_{2n}^0 (\mathcal{O} (X))$. 
\end{Prop}
\begin{proof}
Let $u \in U_{2n}^0 (\mathcal{O} (X)) $ be a null-homotopic unimodular row. By the above Lemma we can extend it to a null-homotopic
symplectic matrix  $A \in \operatorname{Sp^0_{2n}} (\mathcal{O} (X))$. Now we just follow the beginning of the proof of Theorem \ref{t:mainthmrestate}. By Theorem \ref{t:mainthmtoprestate} we can factorize 
$A(x) $ as a product of elementary symplectic matrices with continuous entries. Adding two more elementary symplectic matrices we can achieve that the factorization avoids the singularity set $S_K$. Applying Corollary \ref{c:factorization} we know that $A_0:= A$ is homotopic to a holomorphic map $G=A_1$, via continuous maps $A_t$, such that \[\pi_4(A(x))=\pi_4\circ\Psi_K(A_t(x)), \hskip2mm 0 \le t \le 1 ,\] that is, the last row of the matrices $\Psi_K(A_t(x))$ is constant. Therefore \begin{equation*}
\Psi_K(A_t(x))A(x)^{-1} = \begin{pmatrix} \widetilde{a}_{11,t}(x) & \widetilde{a}_{12,t}(x) & \widetilde{a}_{13,t}(x) & \widetilde{a}_{14,t}(x)    \\   \widetilde{a}_{21,t}(x) & \widetilde{a}_{22,t}(x) & \widetilde{a}_{23,t}(x) & \widetilde{a}_{24,t}(x) 
\\ \widetilde{a}_{31,t}(x) & \widetilde{a}_{32,t}(x) & \widetilde{a}_{33,t}(x) & \widetilde{a}_{34,t}(x) \\ 0 & 0 & 0 & 1
\end{pmatrix}
\end{equation*}
This shows that the element $ \Psi_K(G (x))$ of $\operatorname{Ep_{2n}} (\mathcal{O} (X))$ has the last row equal to $u$ or equivalently
moves the constant row $e_{2n}$ to $u$.
\end{proof}

Let $\psi : \operatorname{SL}_2 \to \operatorname{Sp}_4$ be the standard embedding given by 
\begin{equation}  
\begin{pmatrix}  a & b   \\ 
c & d 
\end{pmatrix} \mapsto \begin{pmatrix}  a & 0 & b & 0    \\   0 & 1 & 0 & 0 
\\ c & 0 & d & 0 \\ 0 & 0 & 0 & 1
\end{pmatrix}
\end{equation} 

Continuing like in  the  proof of Theorem \ref{t:mainthmrestate} we see that it gives the following "inductive step".
\begin{Prop} 
 For a Stein space $X$ holds $$\operatorname{Sp^0_4} (\mathcal{O} (X)) =\operatorname{Ep_{4}} (\mathcal{O} (X)) \cdot \psi ( \operatorname{Sp^0_2} (\mathcal {O} (X))) .$$ 
\end{Prop}

In a similar way one can deduce from our earlier results (\cite{Ivarsson:2012} Proposition 2.8. and proof of Theorem 2.3.) the corresponding statements for
the special linear groups. The definition of the elementary group $\operatorname{E_{n}}$ and the inclusion $\psi$ of $\operatorname{SL}_{n-1}$ into 
$\operatorname{SL}_{n}$ are the usual ones.
\begin{Prop} 
 For a Stein space $X$ and any $n\ge 2$ the group $\operatorname{E_{n}} (\mathcal{O} (X))$ acts transitively on the set of null-homotopic  unimodular rows $U_{n}^0 (\mathcal{O} (X))$. 
\end{Prop}

\begin{Prop} 
 For a Stein space $X$ and any $n\ge 2$ holds $$\operatorname{SL}^0_n (\mathcal{O} (X)) =\operatorname{E_{n}} (\mathcal{O} (X)) \cdot \psi ( \operatorname{SL}^0_{n-1} (\mathcal {O} (X))) .$$ 
\end{Prop}

\section{Stratified sprays}\label{sprays}

We will introduce the concept of a spray associated with a holomorphic submersion following \cite{Gromov:1989} and \cite{Forstneric:2002}. First we introduce some notation and terminology. Let $h\colon Z \to X$ be a holomorphic submersion of a complex manifold $Z$ onto a complex manifold $X$. For any $x\in X$ the fiber over $x$ of this submersion will be denoted by $Z_x$. At each point $z\in Z$ the tangent space $T_zZ$ contains {\it the vertical tangent space} $VT_zZ=\ker Dh$. For holomorphic vector bundles $p\colon E \to Z$ we denote the zero element in the fiber $E_z$ by $0_z$.

\begin{Def} \label{definspray}
Let $h\colon Z \to X$ be a holomorphic submersion of a complex manifold $Z$ onto a complex manifold $X$. A spray on $Z$ associated with $h$ is a triple $(E,p,s)$, where $p\colon E\to Z$ is a holomorphic vector bundle and $s\colon E\to Z$ is a holomorphic map such that for each $z\in Z$ we have 
\begin{itemize}
\item[(i)]{$s(E_z)\subset Z_{h(z)}$,}
\item[(ii)]{$s(0_z)=z$, and}
\item[(iii)]{the derivative $Ds(0_z)\colon T_{0_z}E\to T_zZ$ maps the subspace $E_z\subset T_{0_z}E$ surjectively onto the vertical tangent space $VT_zZ$.}
\end{itemize} 
\end{Def} 

\begin{Rem}
We will also say that the submersion admits a spray. A spray associated with a holomorphic submersion is sometimes called a (fiber) dominating spray.
\end{Rem}

One way of constructing dominating sprays, as pointed out by \textsc{Gromov}, is to find finitely many $\mathbb{C}$-complete vector fields that are tangent to the fibers and span the tangent space of the fibers at all points in $Z$. One can then use the flows $\varphi_j^t$ of these vector fields $V_j$ to define $s\colon Z\times \mathbb{C}^N\to Z$ via $s(z,t_1,\dots, t_N)=\varphi_1^{t_1}\circ \dots \circ \varphi_N^{t_N}(z)$ which gives a spray. 

\begin{Def}
Let $X$ and $Z$ be complex spaces. A holomorphic map $h\colon Z \to X$ is said to be a submersion if for each point $z_0\in Z$ it is locally equivalent via a fiber preserving biholomorphic map to a projection $p\colon U\times V \to U$, where $U\subset X$ is an open set containing $h(z_0)$ and $V$ is an open set in some $\mathbb{C}^d$.  
\end{Def}

We will need to use stratified sprays which are defined as follows.

\begin{Def} 
\label{definstratifiedspray}
We say that a submersion $h\colon Z\to X$ admits stratified sprays if there is a descending chain of closed complex subspaces $X=X_m\supset \cdots \supset X_0$ such that each stratum $Y_k = X_k\setminus X_{k-1}$ is regular and the restricted submersion $h\colon Z|_{Y_k}\to Y_k$ admits a spray over a small neighborhood of any point $x\in Y_k$.  
\end{Def}
\begin{Rem}
We say that the stratification $X=X_m\supset \cdots \supset X_0$ is associated with the stratified spray.
\end{Rem}

In \cite{Forstneric:2001}, see also \cite[Theorem 8.3]{Forstneric:2010}, the following theorem is proved. 
\begin{Thm}\label{t:forstneric}
Let $X$ be a Stein space with a descending chain of closed complex subspaces $X=X_m\supset \cdots \supset X_0$ such that each stratum $Y_k = X_k\setminus X_{k-1}$ is regular. Assume that $h\colon Z \to X$ is a holomorphic submersion which admits stratified sprays then any continuous section $f_0\colon X\to Z$ such that $f_0|_{X_0}$ is holomorphic can be deformed to a holomorphic section $f_1\colon X\to Z$ by a homotopy that is fixed on $X_0$. 
\end{Thm}

\section{Proof of Lemma \ref{l:surjective} and \ref{l:submersive}}
\label{technical1}


Lemma \ref{l:surjective} and \ref{l:submersive} hold for any size matrix. In 
this section we therefore look at $(2n \times 2n)$ matrices. Given two vectors 
$\vec{a}$ and $\vec{b}$ in $\mathbb{C}^{n}$ (i.e. $ n \times 1$ matrices), we denote
by
\[ \begin{pmatrix} \vec{a} \\ \vec{b} \end{pmatrix} \]
the obvious vector in $\mathbb{C}^{2n} $.

We shall consider products of $(2n \times 2n)$-matrices 
\[ \begin{pmatrix} I_n & 0 \\ Z_1 & I_n \end{pmatrix}\begin{pmatrix} I_n & W_1 \\ 0 & I_n \end{pmatrix}\begin{pmatrix} I_n & 0 \\ Z_2 & I_n \end{pmatrix}\begin{pmatrix} I_n & W_2 \\ 0 & I_n   \end{pmatrix}\cdots \]
where $Z_1,Z_2,\cdots$ and $W_1, W_2, \cdots$ are $n \times n$ matrices of variables
\[Z_k = (z_{k,ij}) , W_k= (w_{k,ij}) , 1\le i,j \le n \]
They are symmetric, i.e. $z_{k,ij}=z_{k,ji}$ and $w_{k,ij}=w_{k,ji}$. We call the variables $z_{k,n1},\cdots ,z_{k,nn}$ \emph{last row variables} (this term does not apply to the $w$-variables). If we have $K$ factors, there are $K\frac{n(n+1)}{2}$ variables. We will
also think of the $K$-tuple $(Z_1,W_1,Z_2,W_2,\cdots) $ as a point in  $\mathbb{C}^{K\frac{n(n+1)}{2}}$.
We will study the last row of this product, which is a map $\Phi_K \colon \mathbb{C}^{K\frac{n(n+1)}{2}} \to \mathbb{C}^{2n} \setminus \{0\}$. We prefer to work with the transpose of this row, which we denote by $P^K$, a vector in $\mathbb{C}^{2n}$.
It follows that 
$$
P^1 = \begin{pmatrix} \vec{z} \\ \vec{e_n} \end{pmatrix}
$$
where $\vec{z}=(z_{1,n1},\cdots,z_{1,nn})^T$ and $\vec{e_n}$ is the last standard
basis vector of $\mathbb{C}^{n}$.

The set $S_K$ for $K \ge 2$ is now defined as the set of $K$-tuples of symmetric matrices 
$(Z_1,W_1,\cdots)$ such that in the first $K-1$ matrices all the last row variables
(of the Z's) are $0$ and the columns of all of the W's do not span $\mathbb{C}^{n}$.
(This means the totality of all $W_i$ columns of the $K-1$ first factors.)

\begin{Lem} \label{l:surjection} $P^K : \mathbb{C}^{K\frac{n(n+1)}{2}} \setminus S_K
\rightarrow  \mathbb{C}^{2n} \setminus \{0\} $  is surjective for $K \ge 3$.
\end{Lem}

\begin{proof}
We prove the result for $K=3$. For $K>3$, simply put $W_2=Z_3=W_3=\cdots =0$. 
The proof uses an easy fact from linear algebra; given two vectors $\vec{c}$ and
$\vec{d}$ in $\mathbb{C}^{n}$ with $\vec{c} \ne \vec{0}$ there is a symmetric 
matrix $M$ such that $M\vec{c}=\vec{d}$. Now let
$$
\begin{pmatrix} \vec{a} \\ \vec{b} \end{pmatrix} \in \mathbb{C}^{2n} \setminus \{0\}.
$$
Pick any symmetric matrix $Z_2$ such that $\vec{z}=\vec{a}-Z_2\vec{b} \ne \vec{0}$
and let $Z_1$ be any symmetric matrix whose last row is $\vec{z}$ and $W_1$ a 
symmetric matrix such that $W_1 \vec{z} = \vec{b}-\vec{e_n}$. Then 
$(Z_1,W_1,Z_2) \notin S_3$ and for this choice we have
$$
P^3 = \begin{pmatrix} I_n & Z_{2} \\ 0 & I_n \end{pmatrix}
\begin{pmatrix} I_n & 0 \\ W_1 & I_n \end{pmatrix}
\begin{pmatrix} \vec{z} \\ \vec{e_n} \end{pmatrix}=
\begin{pmatrix} I_n & Z_{2} \\ 0 & I_n \end{pmatrix}
\begin{pmatrix} \vec{z} \\ \vec{b} \end{pmatrix}=
\begin{pmatrix} \vec{a} \\ \vec{b} \end{pmatrix}.
$$

\end{proof}

By a slight abuse of notation, we denote the Jacobian matrix of $\Phi_K$ by $JP^K$. This is a $(2n \times K\frac{n(n+1)}{2})$-matrix whose columns are the derivatives of $P^K$ with respect to one particular variable. We denote the components of $P^K$ by $P^K_i$, $1\le i \le 2n$. It follows that
\begin{equation}\label{e:P2k+1} 
P^{2k+1} = \begin{pmatrix} I_n & Z_{k+1} \\ 0 & I_n \end{pmatrix} P^{2k}
\end{equation}
\begin{equation}\label{e:P2k+2}
P^{2k+2} = \begin{pmatrix} I_n & 0 \\ W_{k+1} & I_n \end{pmatrix} P^{2k+1}
\end{equation}

We shall look at the final part of $JP^{2k+1}$, the part where we differentiate with respect to the new variables $z_{k+1,11},\cdots ,z_{k+1,n1},z_{k+1,22},\cdots, z_{k+2,2n},\cdots, z_{k+1,nn}$. This is a $(2n \times \frac{n(n+1)}{2})$-matrix. The column where we differentiate with respect to $z_{k+1,ij}$ will consist of $P^{2k}_{n+i}$ in row number $j$ and $P^{2k}_{n+j}$ in row number $i$. Hence the bottom half of this matrix is zero and we only look at the upper half, an $(n \times \frac{n(n+1)}{2})$-matrix which we denote by $A_{k+1}$. If we consider just the columns which contain one particular $P^{2k}_{n+i}$, we get a square $(n \times n)$-matrix whose $i$-th row is $(P^{2k}_{n+1},\cdots ,P^{2k}_{2n})$, has $P^{2k}_{n+i}$ along the diagonal and is otherwise zero. The determinant of this submatrix is $(P^{2k}_{n+i})^n$.

The situation is similar for the final part of $JP^{2k+2}$, except now the top half is zero and the bottom half $B_{k+1}$ contains $P^{2k+1}_1, \cdots ,P^{2k+1}_n$ in the same pattern as for $A_{k+1}$.

In the proof of the next lemma it will be convenient to use the following notation: if $A$ and $B$ are two matrices with the same column length, we let $A \cup B $ denote the matrix obtained by extending $A$ with $B$ to the right. $e_{2n}$ denotes the last vector in the standard basis of $\mathbb{C}^{2n} $.

\begin{Lem} \label{l:submersion} 
$P^K$ is a submersion exactly on the set $\mathbb{C}^{K\frac{n(n+1)}{2}} \setminus S_K$.

If $K=2k$ and all the last row variables are zero, then $P^{2k}=e_{2n}$ and the span of the bottom half of the $JP^{2k}$ columns equals the span of the columns of $W_1,W_2,\cdots, W_k $.
\end{Lem}

\begin{proof}
For $N=1$ the theorem is empty. $P^1=(z_{1,n1},\cdots ,z_{1,nn},0,\cdots, 0,1)$ and 
\[JP^1 = \begin{pmatrix} I_n \\ 0 \end{pmatrix} \] where we have removed all zero columns. For $N=2$ we have 
\[P^2 = \begin{pmatrix} I_n & 0 \\ W_1 & I_n \end{pmatrix} P^1 \]
This implies 
\[JP^2 = \begin{pmatrix} I_n & 0 \\ W_1 & I_n \end{pmatrix} \begin{pmatrix} I_n \\ 0 \end{pmatrix} \cup \begin{pmatrix} 0 \\ B_1 \end{pmatrix} = \begin{pmatrix} I_n \\ W_1 \end{pmatrix} \cup \begin{pmatrix} 0 \\ B_1 \end{pmatrix} \]
which has full rank if and only if $B_1$ has full rank. Since $P^1_i = z_{1,ni}$, by the discussion preceding the lemma, $B_1$ has full rank if and only if at least one $z_{1,ni}$ is nonzero.

If all $z_{1,ni}$ are zero, then $P^1=e_{2n}$ and $B_1=0$. Hence the statement about the span is trivially true. 

We now assume that the theorem is true for $N=2k$. We have
\begin{equation}\label{e:JacP2k+1} 
JP^{2k+1} = \begin{pmatrix} I_n & Z_{k+1} \\ 0 & I_n \end{pmatrix} JP^{2k}  \cup \begin{pmatrix} A_{k+1} \\ 0 \end{pmatrix} 
\end{equation}

If at least one of the previous last row variables is nonzero, then $JP^{2k}$ has full rank by the induction hypothesis and so does $JP^{2k+1}$. If not, then $P^{2k}=e_{2n}$ and $A_{k+1}=I_n$, after removing zero columns. If $JP^{2k}= \begin{pmatrix} A \\ B \end{pmatrix} $, then \[JP^{2k+1} = \begin{pmatrix} A+Z_{k+1}B & I_n \\ B & 0 \end{pmatrix} \] which has full rank if and only if $B$ has full rank. But the column span of $B$ equals the column span of $W_1, \cdots, W_k$. This proves the first part of the lemma for $N= 2k+1$.

If all the previous last row variables are zero, it also follows that \[P^{2k+1}=(z_{k+1,n1},\cdots,z_{k+1,nn},0,\cdots,0,1)^t. \] Finally
\begin{equation}\label{e:JacP2k+2} 
JP^{2k+2} = \begin{pmatrix} I_n & 0 \\ W_{k+1} & I_n \end{pmatrix} JP^{2k+1}  \cup \begin{pmatrix} 0 \\ B_{k+1} \end{pmatrix}  
\end{equation}
which has full rank if $JP^{2k+1}$ does. 

If not, then by the above all the previous last row variables are zero and 
\[JP^{2k+2} = \begin{pmatrix} A+Z_{k+1}B & I_ n \\ B+W_{k+1}(A+Z_{k+1}B) & W_{k+1} \end{pmatrix} JP^{2k+1}  \cup \begin{pmatrix} 0 \\ B_{k+1} \end{pmatrix}  \] which has full rank if and only if at least one $z_{k+1,ni}$ is nonzero by the discussion preceding the lemma. This proves the first part of the lemma for $N=2k+2$.

If all the $z_{k+1,ni}$ also are zero, then $P^{2k+1}=e_{2n}$ and so $P^{2k+2}=e_{2n}$. Also $B_{k+1}=0$ and since the columns of
$W_{k+1}(A+Z_{k+1}B)$ are linear combinations of the columns of $W_{k+1}$, the span of the bottom half of $JP^{2k+2}$ equals the span of the columns of $W_1,\cdots,W_{k+1}$ by the induction hypothesis. This completes the proof of the lemma.
\end{proof}

\section{The stratification} \label{s:stratification}
The goal in this section is to describe the stratification needed to understand that the submersion $ \pi_4 \circ \Psi_K \colon ( \mathbb{C}^3 )^K\setminus S_K \to \mathbb{C}^4\setminus \{0\} $ is a stratified elliptic submersion. Let \[ \vec{Z}_K=\begin{cases} (z_1, z_2, z_3, w_1, w_2,w_3,\dots, w_{3k-2},w_{3k-1},w_{3k}) \text{ if } K=2k \\ (z_1, z_2, z_3, w_1, w_2,w_3,\dots, z_{3k+1},z_{3k+2},z_{3k+3}) \text{ if } K=2k+1\end{cases} \] and 
\[ \pi_4 \circ \Psi_K (\vec{Z}_{K}) = \left(P^K_1 (\vec{Z}_{K}), P^K_2 (\vec{Z}_{K}), P^K_3 (\vec{Z}_{K}), P^K_4 (\vec{Z}_{K})\right).\] 
\begin{Rem} \label{r:inclusionPolyRing}
We will abuse notation in the following way in the paper. A polynomial not containing a variable can be interpreted as a polynomial of that variable. More precisely, let $K<L$. We have the projection $\pi\colon \mathbb{C}^L \to \mathbb{C}^K$, $\pi(x_1,\dots, x_L,\dots, x_K)=(x_1,\dots, x_L)$ and $\pi^*\colon \mathbb{C}[ \mathbb{C}^K ] \to \mathbb{C}[ \mathbb{C}^L ]$. For $p\in \mathbb{C}[\mathbb{C}^K]$ we still write $p$ instead of $\pi^*(p)$.  
\end{Rem}
We want to study the fibers 
\[\mathcal{F}^K_{(a_1,a_2,a_3,a_4)} = (\pi_4 \circ \Psi_K)^{-1}(a_1,a_2,a_3,a_4). \]
 Assume first that $ K = 2k+1 \ge 3 $ is odd. We see that 
\[
\pi_4 \circ \Psi_K (\vec{Z}_{K}) = \pi_4 \circ \Psi_{K-1} (\vec{Z}_{K-1}) \begin{pmatrix} 1 & 0 & 0 & 0 \\ 0 & 1 & 0 & 0 \\ z_{3k+1} & z_{3k+2} & 1 & 0 \\ z_{3k+2} & z_{3k+3} & 0 & 1 \end{pmatrix}
\] 
and we get 
\begin{equation*} \begin{aligned} 
P^K_1(\vec{Z}_{K} ) &= P^{K-1}_1(\vec{Z}_{K-1}) + z_{3k+1}  P^{K-1}_3(\vec{Z}_{K-1})+ z_{3k+2}  P^{K-1}_4(\vec{Z}_{K-1}) \\ 
P^K_2(\vec{Z}_{K} ) &= P^{K-1}_2(\vec{Z}_{K-1}) + z_{3k+2}  P^{K-1}_3(\vec{Z}_{K-1})+ z_{3k+3}  P^{K-1}_4(\vec{Z}_{K-1}) \\ 
P^K_3(\vec{Z}_{K} ) &= P^{K-1}_3(\vec{Z}_{K-1}) \\ 
P^K_4(\vec{Z}_{K} ) &= P^{K-1}_4(\vec{Z}_{K-1}).
\end{aligned} \end{equation*} 
We are led to the equations  
\begin{equation} \label{e:oddcase} \begin{aligned} 
a_1 &= P^K_1(\vec{Z}_{K} ) = P^{K-1}_1(\vec{Z}_{K-1}) + z_{3k+1}  P^{K-1}_3(\vec{Z}_{K-1})+ z_{3k+2}  P^{K-1}_4(\vec{Z}_{K-1}) \\ 
a_2 &= P^K_2(\vec{Z}_{K} ) = P^{K-1}_2(\vec{Z}_{K-1}) + z_{3k+2}  P^{K-1}_3(\vec{Z}_{K-1})+ z_{3k+3}  P^{K-1}_4(\vec{Z}_{K-1}) \\ 
a_3 &= P^K_3(\vec{Z}_{K} ) = P^{K-1}_3(\vec{Z}_{K-1}) \\ 
a_4 &= P^K_4(\vec{Z}_{K} ) = P^{K-1}_4(\vec{Z}_{K-1}).
\end{aligned} \end{equation}
Notice that these equations simplify to 
\[ \begin{aligned} 
a_1 &= P^{K-1}_1(\vec{Z}_{K-1}) + a_3 z_{3k+1} + a_4 z_{3k+2} \\ 
a_2 &= P^{K-1}_2(\vec{Z}_{K-1}) + a_3 z_{3k+2} + a_4 z_{3k+3}  \\ 
a_3 &= P^{K-1}_3(\vec{Z}_{K-1}) \\ 
a_4 &= P^{K-1}_4(\vec{Z}_{K-1}).
\end{aligned} \] 
If $ (a_3,a_4)\neq (0,0) $ then we can solve the two first equations for two of the three variables $ z_{3k+1},z_{3k+2},z_{3k+3} $ and we see that the fiber is a graph over $ \mathcal{G}^{K-1}_{(a_3,a_4)} \times \mathbb{C} $ where
\[ 
\mathcal{G}^{K-1}_{(a_3,a_4)} =\left\{\vec{Z}_{K-1}\in \mathbb{C}^{3K-3}; a_3 = P^{K-1}_3(\vec{Z}_{K-1}), 
a_4 = P^{K-1}_4(\vec{Z}_{K-1})\right\}. 
\]  
If $ (a_3,a_4)= (0,0) $ we get $\mathcal{F}^K_{(a_1,a_2,0,0)}=\mathcal{F}^{K-1}_{(a_1,a_2,0,0)} \times \mathbb{C}^3  $. We see that we get two main cases, namely $ (a_3,a_4)= (0,0) $ and $ (a_3,a_4) \neq (0,0) $. The last case will break into the two sub-cases, namely $ (a_3,a_4) \neq (0,1) $ and $ (a_3,a_4)= (0,1) $. We need these sub-cases because $ \mathcal{G}^{K-1}_{(0,1)} $ is not smooth. We list the strata below: \begin{itemize}  
\item The strata of generic fibers: When $ (a_3,a_4) \neq (0,0) $. The fibers are graphs over $ \mathcal{G}^{K-1}_{(a_3,a_4)} \times \mathbb{C} $. This set is divided into two strata as follows:
\begin{itemize}
\item   Smooth generic fibers:  When $ (a_3,a_4) \neq (0,1) $ then the fibers are smooth.
\item  Singular generic fibers: When $ (a_3,a_4) = (0,1) $ then the fibers are non-smooth.
\end{itemize} 
\item The stratum of non-generic fibers: When $ (a_3,a_4) = (0,0) $ the fibers are $\mathcal{F}^K_{(a_1,a_2,0,0)}=\mathcal{F}^{K-1}_{(a_1,a_2,0,0)} \times \mathbb{C}^3  $. Moreover the fibers are smooth.
 \end{itemize} 
We now analyse the case when $ K=2k \ge 3 $ is even. Now we have
\[
\pi_4 \circ \Psi_K (\vec{Z}_{K}) = \pi_4 \circ \Psi_{K-1} (\vec{Z}_{K-1}) \begin{pmatrix}  1 & 0 & w_{3k-2} & w_{3k-1}  \\ 0 & 1 & w_{3k-1} & w_{3k} \\ 0 & 0 & 1 & 0 \\ 0 & 0 & 0 & 1  \end{pmatrix}
\] 
and $ \mathcal{F}^K_{(a_1,a_2,a_3,a_4)} $ is the solution set to the equations 
\begin{equation} \label{e:evencase} \begin{aligned} 
a_1 & =  P^K_1(\vec{Z}_{K} ) = P^{K-1}_1(\vec{Z}_{K-1}) \\ 
a_2 & = P^K_2(\vec{Z}_{K} ) = P^{K-1}_2(\vec{Z}_{K-1})\\ 
a_3 & = P^K_3(\vec{Z}_{K} ) = P^{K-1}_3(\vec{Z}_{K-1})  + w_{3k-2}  P^{K-1}_1(\vec{Z}_{K-1})+ w_{3k-1}  P^{K-1}_2(\vec{Z}_{K-1})\\ 
a_4 & = P^K_4(\vec{Z}_{K} ) = P^{K-1}_4(\vec{Z}_{K-1})  + w_{3k-1}  P^{K-1}_1(\vec{Z}_{K-1})+ w_{3k}  P^{K-1}_2(\vec{Z}_{K-1}) .
\end{aligned} \end{equation} 
As in the previous case these equations simplify
\[ \begin{aligned} 
a_1 & = P^{K-1}_1(\vec{Z}_{K-1}) \\ 
a_2 & = P^{K-1}_2(\vec{Z}_{K-1})\\ 
a_3 & = P^{K-1}_3(\vec{Z}_{K-1})  + a_1 w_{3k-2} + a_2 w_{3k-1} \\ 
a_4 & = P^{K-1}_4(\vec{Z}_{K-1})  + a_1 w_{3k-1} + a_2 w_{3k}.
\end{aligned} \] 
Let
\[ 
\mathcal{H}^{K-1}_{(a_1,a_2)} =\left\{\vec{Z}_{K-1}\in \mathbb{C}^{3K-3}; a_1 = P^{K-1}_1(\vec{Z}_{K-1}), 
a_2 = P^{K-1}_2(\vec{Z}_{K-1})\right\}. 
\]  
A similar analysis as in the previous gives us the following strata:
\begin{itemize}  
\item The stratum of generic fibers: When $ (a_1,a_2) \neq (0,0) $. The fibers are graphs over $ \mathcal{H}^{K-1}_{(a_1,a_2)} \times \mathbb{C} $. Moreover the fibers are smooth. 
\item The strata of non-generic fibers: When $ (a_1,a_2) = (0,0) $ the fibers are $\mathcal{F}^K_{(0,0,a_3,a_4)}=\mathcal{F}^{K-1}_{(0,0,a_3,a_4)} \times \mathbb{C}^3  $. This set is divided into two strata as follows:
\begin{itemize}
\item   Smooth non-generic fibers:  When $ (a_3,a_4) \neq (0,1) $ then the fibers are smooth.
\item  Singular non-generic fibers: When $ (a_3,a_4) = (0,1) $ then the fibers are non-smooth.
\end{itemize} 
 \end{itemize} 

\section{Determination of complete vector fields}\label{s:descComplete}


The description of the fibers in Section \ref{s:stratification} leads us to study vector fields simultaneously tangent to the level sets $ \{P=c_1\}, \{Q=c_2\} $ of two functions $ P,Q\colon \mathbb{C}^N\to \mathbb{C} $. Such fields can be constructed in the following way. Pick three variables $ x,y,z $ from the variables $ x_1,\dots, x_N $ on $ \mathbb{C}^N $ and consider the vector fields \begin{equation}\label{e:fields} D_{xyz}(P,Q)=\det \begin{pmatrix}  {\partial}/{\partial x} & {\partial}/{\partial y} & {\partial}/{\partial z}\\ {\partial P}/{\partial x} & {\partial P}/{\partial y} & {\partial P}/{\partial z} \\ {\partial Q}/{\partial x} & {\partial Q}/{\partial y} & {\partial Q}/{\partial z}\end{pmatrix} \end{equation} which are simultaneously tangent to the level sets.  As mentioned in Section \ref{sprays} we want to use a finite collection of complete vector fields spanning tangent space at every point to prove (stratified) ellipticity. It is an easy exercise to show that the collection of these vector fields over all possible triples spans the tangent space at smooth points of  the variety $ \{P=c_1\} \cap \{Q=c_2\} $. It turns out that many of the vector fields we get by this method are complete but unfortunately not all of them. The complete vector fields from this collection will not span the tangent space at all points for all level sets. To overcome this difficulty and still producing  dominating sprays from  this collection of available complete fields is the main technical part of our paper explained in Section \ref{s:strategy}.

Now we will start to describe the complete vector fields tangent to the fibers of $ \pi_4\circ\Psi_K=(P_1^K,P_2^K, P_3^K, P_4^K) $ that we get using (\ref{e:fields}).
It will be convenient to group the variables as in Section \ref{technical1}, $Z_1, W_1,Z_2,W_2,\cdots $, where $$Z_k =  \begin{pmatrix} z_{3k-2} & z_{3k-1} \\ z_{3k-1} & z_{3k} \end{pmatrix} $$ and similarly for $W_k$. Since the variable $z_1$ never enters 
in $P^K$, we omit it from the first group $Z_1$. Note that $P^1=(z_2,z_3,0,1)^T$. We are
going to study the vector fields $$V_{ij}^K (x,y,z)= D_{xyz}(P_i^K,P_j^K) . $$ The 
$2 \times 2$ minors occuring as coefficients are denoted by $C_{ij}^K (\cdot,\cdot)$, i.e.
$$V_{ij}^K (x,y,z)= C_{ij}^K (y,z){\partial}/{\partial x} - C_{ij}^K (x,z){\partial}/{\partial y} + C_{ij}^K (x,y){\partial}/{\partial z}. $$

The description of the complete vector fields will be done inductively. We start with $ K=2 $. We have to study $ \mathcal{G}_{(a_3,a_4)}^2 $ or equivalently the equations \begin{equation} \begin{aligned} 
a_3 &= P_3^2(z_1,\dots,w_3)= z_2w_1+z_3w_2 \\
a_4 &= P_4^2(z_1,\dots,w_3)= 1+z_2w_2+z_3w_3.
\end{aligned}\end{equation}
We are interested in which triples $(x,y,z)$ of variables from the list $z_2, z_3, w_1, 
w_2, w_3$ give complete vector fields $V_{34}^2(x,y,z) $ and we denote the set of these triples by 
$\mathcal{T}_2$. By definition $\mathcal{T}_1 = \emptyset$. 
An easy computation gives that
\begin{equation}\label{e:T2}
\begin{aligned}\mathcal{T}_2 = \{&(w_1,w_2,w_3),(z_2,w_2,w_3),(z_3,w_1,w_2), \\ & (z_2,w_1,w_3),(z_3,w_1,w_3),(z_2,z_3,w_1),(z_2,z_3,w_3)\}.\end{aligned}
\end{equation}
For all the remaining noncomplete triples there is a variable such that the equation
is quadratic for that variable. We are now interested in determining at every stage
the triples of variables $(x,y,z)$ such that $V_{12}^{2k+1}(x,y,z) $ for $K=2k+1$ odd
is complete and $V_{34}^{2k+2}(x,y,z) $ for $K=2k+2$ even. We shall denote the set of
such triples by $\mathcal{T}_K$. The terms occuring in $P^K$ are of degree one in 
the occuring variables, hence the coefficients $C^K_{ij}$  are either of degree one or
two in the occuring variables. A triple giving a coefficient which is quadratic in
the integration variable  (for instance if $C_{ij}^K(y,z) $ is quadratic in the $x$ 
variable) will not be complete and we shall refer to such a triple as a quadratic
triple and the corresponding vector field as a quadratic vector field. The content of the next lemma is that all the remaining triples give 
complete vector fields. The variables that do not occur in a triple will have constant
solutions and are therefore treated as such in the proof.



\begin{Lem} \label{l:completetriple}
For $k\ge 1$ we have $\mathcal{T}_{2k}\subset \mathcal{T}_{2k+1} \subset \mathcal{T}_{2k+2}$. Moreover \begin{equation}\label{e:oddminuseven}\begin{aligned}
    &\mathcal{T}_{2k+1}  \setminus \mathcal{T}_{2k}=\{(z_{3k+1},z_{3k+2},z_{3k+3})\} \cup \phantom{i} \\ \phantom{i} &\cup \{(w_{3k-2},z_{3k+1},z_{3k+3}),(w_{3k-2},z_{3k+2},z_{3k+3})\} \cup \phantom{a}  \\  \phantom{a} & \cup \{(w_{3k},z_{3k+1},z_{3k+2}),(w_{3k},z_{3k+1},z_{3k+3})\} \cup \phantom{a} \\ \phantom{a} & \cup \{(a,b,z_{3k+1}),(a,b,z_{3k+3}); a \text{ and } b \text{ are from the same group}  \} \cup \phantom{a} \\
    \phantom{a} & \cup \{(a,b,z_{3k+1}); a \text{ the last variable of one group and } b \text{ the first of the next} \} \cup \phantom{a} \\
    \phantom{a} & \cup \{(a,b,z_{3k+3}); a \text{ the last variable of one group and } b \text{ the first of the next} \} \cup \phantom{a}\\
    \phantom{a} & \cup \{(a,b,z_{3k+1}); a \text{ the first variable of one group and } b \text{ the last of the next} \} \cup \phantom{a} \\
    \phantom{a} & \cup \{(a,b,z_{3k+3}); a \text{ the first variable of one group and } b \text{ the last of the next} \} 
    \end{aligned}
\end{equation} 
and 
\begin{equation}\label{e:evenminusodd} \begin{aligned}
    & \mathcal{T}_{2k+2}  \setminus \mathcal{T}_{2k+1}= \{(w_{3k+1},w_{3k+2},w_{3k+3})\} \cup \phantom{i} \\ \phantom{i} & \cup \{(z_{3k+1},w_{3k+1},w_{3k+3}),(z_{3k+1},w_{3k+2},w_{3k+3})\} \cup \phantom{i}  \\  \phantom{i} & \cup \{(z_{3k+3},w_{3k+1},w_{3k+2}),(z_{3k+3},w_{3k+1},w_{3k+3})\} \cup \phantom{i} \\ \phantom{i} & \cup \{(a,b,w_{3k+1}),(a,b,w_{3k+3}); a \text{ and } b \text{ are from the same group}  \} \cup \phantom{i} \\
    \phantom{i} & \cup \{(a,b,w_{3k+1}); a \text{ the last variable of one group and } b \text{ the first of the next} \} \cup \phantom{i} \\
    \phantom{i} & \cup \{(a,b,w_{3k+3}); a \text{ the last variable of one group and } b \text{ the first of the next} \} \cup \phantom{i}\\
    \phantom{i} & \cup \{(a,b,w_{3k+1}); a \text{ the first variable of one group and } b \text{ the last of the next} \} \cup \phantom{i} \\
    \phantom{i} & \cup \{(a,b,w_{3k+3}); a \text{ the first variable of one group and } b \text{ the last of the next} \}. 
    \end{aligned}
\end{equation} In combination with (\ref{e:T2}) this gives us a complete description of the sets $\mathcal{T}_L$, $L\ge 2$.
\end{Lem}
\begin{proof}
The result is true for $\mathcal{T}_2$. The first group is interpreted as $\{ z_2,z_3 \}$ 
and $z_{3k+1}$ must be replaced by $z_2$. The missing triplet are precisely the quadratic
triples. 

We shall prove (\ref{e:oddminuseven}), the proof of (\ref{e:evenminusodd}) being identical.
There is a lot of symmetry in the proof and we will not repeat arguments already given
in a situation symmetric to a proven statement. We first consider triples $(x,y,z)$ not
containing any variables from the new group $Z_{k+1}$, i.e. $z_{3k+1},z_{3k+2}$ and 
$z_{3k+3}$. It then follows from (\ref{e:P2k+1}) (omitting variables for shorter notation) 
that:
\begin{equation}\label{e:V12} 
V_{12}^{2k+1}= V_{12}^{2k}+z_{3k+1}V_{32}^{2k}-z_{3k+2}V_{24}^{2k}-z_{3k+2}V_{31}^{2k}
+z_{3k+3}V_{14}^{2k}+(z_{3k+1}z_{3k+3}-z_{3k+2}^2)V_{34}^{2k}.
\end{equation}
A quadratic triple will still be quadratic since $V_{34}^{2k}$ is. For a triple in
$\mathcal{T}_{2k}$, notice that in all the first 5 terms the $V_{ij}^{2k}$ is obtained
by replacing one or two of the functions $P_3^{2k}$ and $P_4^{2k}$ by $P_1^{2k}$ 
and/or $P_2^{2k}$. By (\ref{e:P2k+2}) all of the terms occuring in $P_1^{2k}$
or $P_2^{2k}$ divide  a term occuring in $P_3^{2k}$ and also a term occuring in
$P_4^{2k}$. This means that all terms occuring in the 5 first vector fields above are 
already present in $V_{34}^{2k}$ and completeness is not destroyed. We also notice 
that for any pair $x,y$ of previous variables, the coefficient $C_{12}^{2k+1}(x,y)$ will
also satisfy (\ref{e:V12}).

We next consider triples containing some of the new variables $z_{3k+1},z_{3k+2}$ and 
$z_{3k+3}$. The Jacobian matrix is now given by (\ref{e:JacP2k+1}) where  
\begin{equation}\label{e:Ak+1} 
A^{k+1} = \begin{pmatrix} P_3^{2k} & P_4^{2k} & 0 \\ 0 & P_3^{2k} & P_4^{2k} \end{pmatrix}. 
\end{equation}

If the triple contains all three variables, then
$$ V_{12}^{2k+1}(z_{k+1},z_{k+2},z_{k+3}) = (P_4^{2k})^2 \partial / \partial z_{3k+1}
- (P_3^{2k})(P_4^{2k}) \partial / \partial z_{3k+2} + (P_3^{2k})^2
\partial / \partial z_{3k+3}
$$
and the coefficients do not contain any of the $Z_{k+1}$ variables, hence this is complete.
(The solutions are just affine functions.)

We now consider the case of two new variables. The first possibility is $(x,z_{3k+1},z_{3k+2})$. The coefficient of $\partial / \partial x$ is $(P_3^{2k})^2$. 
Since $P_3^{2k}$ contains all previous variables except $w_{3k}$, this is quadratic in
all those variables and $x=w_{3k}$ is the only possibility. The solution for $w_{3k}$
is affine. The coefficient of $\partial / \partial z_{3k+2}$ is now
$$
-(\frac{\partial P_2^{2k}}{\partial w_{3k}}+z_{k+3}\frac{\partial P_4^{2k}}{\partial w_{3k}})
$$
which is just a constant and the solution is again affine. Finally the coefficient of
$\partial / \partial z_{3k+1}$ is given by
$$
\frac{\partial P_1^{2k}}{\partial w_{3k}}+z_{k+2}\frac{\partial P_4^{2k}}{\partial w_{3k}}
$$
which is an affine function and the solution is entire. Hence this field is complete.

The precise same logic applies to the triple $(x,z_{3k+2},z_{3k+3})$ except now 
$w_{3k-2}$ is the only missing variable (now in $P_4^{2k}$). 

The final possibility of two new variables is the triple $(x,z_{3k+1},z_{3k+3})$.
The coefficient of $\partial / \partial x$ is now $P_3^{2k} P_4^{2k}$ which is of 
degree one in $w_{3k-2}$ and $w_{3k}$ and quadratic in all other previous variables.
We consider the case of $x=w_{3k-2}$, the case $x=w_{3k}$ being identical. The 
coefficient is an affine function of $w_{3k-2}$, hence the solution is entire.
The coefficient of $\partial / \partial z_{3k+1}$ is $-z_{3k+1}P_1^{2k-1}P_4^{2k}$
which is just a linear function of $z_{3k+1}$ and the solution is entire. 
The coefficient of $\partial / \partial z_{3k+3}$ is $-z_{3k+2}P_2^{2k-1}P_4^{2k}$
which is just a constant and the solution is affine. 

We finally consider the case of one new variable and two previous variables $x,y$. It
follows that $C_{12}^{2k+1}(x,y)$ satifies equation (\ref{e:V12}) hence is quadratic
in $z_{3k+2}$, so this cannot be the new variable. In order to investigate $z_{3k+1}$
and $z_{3k+3}$ we need to understand which variables are involved in the coefficients.
To do this we look at each previous group of variables $Z_j$ and $W_j$ for $1 \le j \le k$
and see which variables are involved in the first two rows of the Jacobian with
respect to these variables at level $2k+1$. For a $Z_j$ group we need to consider
the matrix
$$
\begin{pmatrix} \frac{\partial P_1^{2k+1}}{\partial z_{3j-2}} 
& \frac{\partial P_1^{2k+1}}{\partial z_{3j-1}} 
& \frac{\partial P_1^{2k+1}}{\partial z_{3j}} 
\\ \frac{\partial P_2^{2k+1}}{\partial z_{3j-2}} 
& \frac{\partial P_2^{2k+1}}{\partial z_{3j-1}} 
& \frac{\partial P_2^{2k+1}}{\partial z_{3j}} \end{pmatrix}
$$
and the same for a $W_j$ group. The $Z_1$ group only consists of $z_2$ and $z_3$.
The $Z_j$ variables do not occur in the above matrix. There is a simple 
formula for the above matrix which follows from (\ref{e:JacP2k+1}) and (\ref{e:JacP2k+2}).
The matrix is the first two rows of the matrix ($I=I_2$) :
$$
\begin{pmatrix} I & Z_{k+1} \\ 0 & I \end{pmatrix} \cdots
\begin{pmatrix} I & 0 \\ W_j & I \end{pmatrix}
\begin{pmatrix} A_j \\ 0  \end{pmatrix}
$$
and this formula makes it easy to track which variables are missing at each step,
in addition to the $Z_j$ variables. We arrive at the following matrix of missing
variables
$$
\begin{pmatrix}
w_{3j-3},w_{3j},z_{3k+3} & z_{3k+3} & w_{3j-5},w_{3j-2},z_{3k+3} \\
w_{3j-3},w_{3j},z_{3k+1} & z_{3k+1} & w_{3j-5},w_{3j-2},z_{3k+1}
\end{pmatrix}
$$
In the case $j=1$ the missing variable matrix is
$$
\begin{pmatrix}
w_3,z_{3k+3} & w_1,z_{3k+3} \\
w_3,z_{3k+1} & w_1,z_{3k+1}
\end{pmatrix}.
$$

We now consider a $W_j$ group. Again the $W_j$ variables do not enter. We
now have to consider the two first rows of the matrix
$$
\begin{pmatrix} I & Z_{k+1} \\ 0 & I \end{pmatrix} \cdots
\begin{pmatrix} I & Z_{j+1} \\ 0 & I \end{pmatrix}
\begin{pmatrix} 0 \\ B_j  \end{pmatrix}
$$
and this leads to the following missing variables matrix for $j<k$:
$$
\begin{pmatrix}
z_{3j},z_{3j+3},z_{3k+3} & z_{3k+3} & z_{3j-2},z_{3j+1},z_{3k+3} \\
z_{3j},z_{3j+3},z_{3k+1} & z_{3k+1} & z_{3j-2},z_{3j+1},z_{3k+1}
\end{pmatrix}.
$$
For $j=1$ we replace $z_{3j-2}$ by $z_2$. For $j=k$ the middle entries in the upper
left and the lower right corners are replaced by $z_{3k+2}$.

We first investigate triples $(x,y,z_{3k+1})$, where $x$ and $y$ are not from 
$Z_{k+1}$. If $x$ and $y$ are from the same group, then since $z_{3k+1}$ occurs
in every entry in the second row of the missing variable matrix, 
$C_{12}^{2k+1}(x,z_{3k+1})$ and $C_{12}^{2k+1}(y,z_{3k+1})$ do not depend on any
of the variables $x,y,z_{3k+1}$ hence $x$ and $y$ are both affine functions.
$C_{12}^{2k+1}(x,y)$ does not depend on $x,y$ and is of degree one in $z_{3k+1}$,
hence the solution is entire.

Now assume that $x$ and $y$ are from different groups. If $x$ is not a missing
variable in $\frac{\partial P_2^{2k+1}}{\partial y}$, then $y$ is not a missing
variable in $\frac{\partial P_2^{2k+1}}{\partial x}$. $x$ and $y$ are not both
$w_{3k}$, let's say $x$. Then 
$$
C_{12}^{2k+1}(y,z_{3k+1}) = -(\frac{\partial P_2^{2k+1}}{\partial y})P_3^{2k}
$$
is quadratic in $x$ and the field is not complete. Hence $x$ and $y$ must both 
appear in the second row of the missing variable matrix of each other. 

We now look at possibilities for $x$ and $y$. Assume first that $x$ is in $Z_j$
group with $1 < j \le k$. There are now four possibilities ; $x=z_{3j-2}$ in which 
case $y= w_{3j-3}$ or $y=w_{3j}$ or $x=z_{3j}$ in which case $y=w_{3j-5}$ or 
$y=w_{3j-2}$. We consider the first case. Then
$$
C_{12}^{2k+1}(w_{3j-3},z_{3j-2})= \frac{\partial P_1^{2k+1}}{\partial w_{3j-3}}
\frac{\partial P_2^{2k+1}}{\partial z_{3j-2}}-\frac{\partial P_2^{2k+1}}{\partial w_{3j-3}}
\frac{\partial P_1^{2k+1}}{\partial z_{3j-2}}
$$
and from the missing variable matrix we see that this does not depend on $z_{3j-2}$
and $w_{3j-3}$ and is of degree one in $z_{3k+1}$, hence we have an entire solution
for $z_{3k+1}$.
We also have
$$
C_{12}^{2k+1}(w_{3j-3},z_{3k+1})=
- \frac{\partial P_2^{2k+1}}{\partial w_{3j-3}}P_3^{2k}
$$
$$
C_{12}^{2k+1}(z_{3j-2},z_{3k+1})=
- \frac{\partial P_2^{2k+1}}{\partial z_{3j-2}}P_3^{2k}.
$$
The partial derivatives on the right hand sides do not depend on any of the
variables in the triple, hence are just constants. It also follows from the
missing variable matrix that $P_3^{2k}$ does not contain the product of 
$z_{3j-2}$ and $w_{3j-3}$, hence the equations for these two variables is a
linear system with constant coefficients. This has an entire solution.
The three other cases all have similar structure and have entire solutions.
In the case $j=1$, we either have $x=z_2$ and $y=w_3$ or $x=z_3$  and $y=w_1$
and the discussion is the same. It also follows from the missing variable
matrix that $x$ and $y$ cannot come from different $W$ groups. This proves
the result in the case of picking $z_{3k+1}$ from the last group. The proof
in case of picking $z_{3k+3}$ from the last group is completely symmetric.
This provides the final detail in the proof.
\end{proof}

In order to produce complete fields that are also tangential to fibers of the submersion we introduce the following notation and terminology.

\begin{Def}
    Let $\Xi_3=\mathcal{T}_{2}$. For $K\ge 4$ let \[\Xi_K=\mathcal{T}_{K-1}\setminus \mathcal{T}_{K-2}.\] We say that the triples in $\Xi_K$ are {\it introduced on level} $K$.   
\end{Def}

We will now use these complete fields to produce complete fields which are tangential to the fibers $ \mathcal{F}_{(a_1,a_2,a_3,a_4)}^K $. Here we will use triples introduced on level $K$ to produce complete tangential fields.   

First consider the case $ K=2k+1\ge 3 $ odd.

If $ \mathbf{a_3\neq 0} $ we use (\ref{e:oddcase}) to get \[z_{3k+2}=\frac{1}{a_3}\left( a_2-P_{2}^{2k}(\vec{Z}_{2k})-a_4z_{3k+3}\right)\] and  \begin{equation*}\begin{aligned}z_{3k+1}&=\frac{1}{a_3}\left( a_1-P_{1}^{2k}(\vec{Z}_{2k})-a_4z_{3k+2}\right)= \\ &=\frac{1}{a_3}\left( a_1-P_{1}^{2k}(\vec{Z}_{2k})-\frac{a_4}{a_3}\left( a_2-P_{2}^{2k}(\vec{Z}_{2k})-z_{3k+3}P_{4}^{2k}(\vec{Z}_{2k})\right)\right)= \\ &=\frac{1}{a_3^2}\left( a_1a_3-a_3P_{1}^{2k}(\vec{Z}_{2k})-a_2a_4+a_4P_{2}^{2k}(\vec{Z}_{2k})+a_4^2z_{3k+3}\right). \end{aligned}\end{equation*} Using this we define a biholomorphism \[ \alpha\colon \mathcal{G}_{(a_3,a_4)}^{2k}\times \mathbb{C}_{z_{3k+3}} \to \mathcal{F}_{(a_1,a_2,a_3,a_4)}^K\] On \[ \mathcal{G}_{(a_3,a_4)}^{2k}\times \mathbb{C}_{z_{3k+3}}\] we have the complete fields $ \partial_{x_1x_2x_3}^{2k} $ for $ x_1,x_2,x_3 $ in $ \Xi_{2k+1} $ and also the complete field $ \partial / \partial z_{3k+3} $. Using the biholomorphism $ \alpha $ we get complete fields on $ \mathcal{F}_{(a_1,a_2,a_3,a_4)}^K $ for $ a_3 \neq 0 $ of the form \begin{equation} \begin{aligned}\theta_{x_1x_2x_3}^{2k+1,*}&=\partial_{x_1x_2x_3}^{2k}+\frac{1}{a_3}\partial_{x_1x_2x_3}^{2k}\left(P_{2}^{2k}(\vec{Z}_{2k})\right)\frac{\partial}{\partial z_{3k+2}}+\\ &+\frac{1}{a_3^2}\partial_{x_1x_2x_3}^{2k}\left(a_3P_{1}^{2k}(\vec{Z}_{2k})-a_4P_{2}^{2k}(\vec{Z}_{2k})\right) \frac{\partial}{\partial z_{3k+1}} \end{aligned}\end{equation} and \begin{equation} \gamma^{2k+1,*}=\frac{\partial}{\partial z_{3k+3}}+\frac{a_4^2}{a_3^2}\frac{\partial}{\partial z_{3k+1}}-\frac{a_4}{a_3}\frac{\partial}{\partial z_{3k+2}} \end{equation} Since $ P_3^{2k} = a_3$ and $ P_4^{2k} = a_4$ on the fiber we get meromorphic fields on $ (\mathbb{C}^3)^K $ \begin{equation} \begin{aligned}\theta_{x_1x_2x_3}^{2k+1,*}&=\partial_{x_1x_2x_3}^{2k}+\frac{\partial_{x_1x_2x_3}^{2k}\left(P_{2}^{2k}(\vec{Z}_{2k})\right)}{P_{3}^{2k}(\vec{Z}_{2k})}\frac{\partial}{\partial z_{3k+2}}+\\ &+\left(\frac{\partial_{x_1x_2x_3}^{2k}\left(P_{1}^{2k}(\vec{Z}_{2k})\right)}{P_{3}^{2k}(\vec{Z}_{2k})}-\frac{P_{4}^{2k}(\vec{Z}_{2k})\partial_{x_1x_2x_3}^{2k}\left(P_{2}^{2k}(\vec{Z}_{2k})\right)}{P_{3}^{2k}(\vec{Z}_{2k})^2}\right) \frac{\partial}{\partial z_{3k+1}} \end{aligned}\end{equation} and \begin{equation} \gamma^{2k+1,*}=\frac{\partial}{\partial z_{3k+3}}+\frac{P_{4}^{2k}(\vec{Z}_{2k})^2}{P_{3}^{2k}(\vec{Z}_{2k})^2}\frac{\partial}{\partial z_{3k+1}}-\frac{P_{4}^{2k}(\vec{Z}_{2k})}{P_{3}^{2k}(\vec{Z}_{2k})}\frac{\partial}{\partial z_{3k+2}} \end{equation} (abusing notation) with poles on $ P_3^{2k} = 0$. Since $ P_3^{2k} $ is in the kernel of these fields we can multiply the fields by $ (P_3^{2k})^2 $ and get complete fields that are globally defined on ($\mathbb{C}^3)^K $ and preserve the fibers of $ \pi_4\circ \Psi_K $ below \begin{equation} \label{e:thetafieldodd} \begin{aligned}\theta_{x_1x_2x_3}^{2k+1}&=P_3^{2k}(\vec{Z}_{2k})^2\theta_{x_1x_2x_3}^{2k+1,*}=P_3^{2k}(\vec{Z}_{2k})^2\partial_{x_1x_2x_3}^{2k}+\\&+P_3^{2k}(\vec{Z}_{2k})\partial_{x_1x_2x_3}^{2k}\left(P_{2}^{2k}(\vec{Z}_{2k})\right)\frac{\partial}{\partial z_{3k+2}}+\\ &+\Big[P_3^{2k}(\vec{Z}_{2k})\partial_{x_1x_2x_3}^{2k}\left(P_{1}^{2k}(\vec{Z}_{2k})\right)-\\&-P_{4}^{2k}(\vec{Z}_{2k})\partial_{x_1x_2x_3}^{2k}\left(P_{2}^{2k}(\vec{Z}_{2k})\right)\Big] \frac{\partial}{\partial z_{3k+1}} \end{aligned} \end{equation} for $ x_1,x_2,x_3 \in \Xi_{2k+1} $ and the field \begin{equation}\label{e:gammaodd} \begin{aligned}
\gamma^{2k+1}&=P_3^{2k}(\vec{Z}_{2k})^2\gamma^{2k+1,*}=P_3^{2k}(\vec{Z}_{2k})^2\frac{\partial}{\partial z_{3k+3}}+\\&+P_{4}^{2k}(\vec{Z}_{2k})^2\frac{\partial}{\partial z_{3k+1}}-P_3^{2k}(\vec{Z}_{2k})P_{4}^{2k}(\vec{Z}_{2k})\frac{\partial}{\partial z_{3k+2}}.\end{aligned}\end{equation}

If $ \mathbf{a_4\neq 0} $ we can define a biholomorphism \[ \beta\colon \mathcal{G}_{(a_3,a_4)}^{2k}\times \mathbb{C}_{z_{3k+1}} \to \mathcal{F}_{(a_1,a_2,a_3,a_4)}^K \] using (\ref{e:oddcase}) and \[z_{3k+2}=\frac{1}{a_4}\left( a_1-P_{1}^{2k}(\vec{Z}_{2k})-a_3z_{3k+1}\right)\] and  \begin{equation*}\begin{aligned}z_{3k+3}&=\frac{1}{a_4}\left( a_2-P_{2}^{2k}(\vec{Z}_{2k})-a_3z_{3k+2}\right)= \\ &=\frac{1}{a_4}\left( a_2-P_{2}^{2k}(\vec{Z}_{2k})-\frac{a_3}{a_4}\left( a_1-P_{1}^{2k}(\vec{Z}_{2k})-a_3z_{3k+1}\right)\right)= \\ &=\frac{1}{a_4^2}\left( a_2a_4-a_4P_{2}^{2k}(\vec{Z}_{2k})-a_1a_3+a_3P_{1}^{2k}(\vec{Z}_{2k})+a_3^2z_{3k+1}\right). \end{aligned}\end{equation*} On \[\mathcal{G}_{(a_3,a_4)}^{2k}\times \mathbb{C}_{z_{3k+1}}\] we have the complete fields $ \partial_{x_1x_2x_3}^{2k} $ for $ x_1,x_2,x_3 $ in $ \Xi_{2k+1} $ and $ \partial / \partial z_{3k+1} $. Proceeding as above we get the complete fields \begin{equation} \label{e:phifieldodd}  \begin{aligned}\phi_{x_1x_2x_3}^{2k+1}&=P_4^{2k}(\vec{Z}_{2k})^2\partial_{x_1x_2x_3}^{2k}+\\&+P_4^{2k}(\vec{Z}_{2k})\partial_{x_1x_2x_3}^{2k}\left(P_{1}^{2k}(\vec{Z}_{2k})\right)\frac{\partial}{\partial z_{3k+2}}+\\ &+\Big[P_4^{2k}(\vec{Z}_{2k})\partial_{x_1x_2x_3}^{2k}\left(P_{2}^{2k}(\vec{Z}_{2k})\right)-\\&-P_{3}^{2k}(\vec{Z}_{2k})\partial_{x_1x_2x_3}^{2k}\left(P_{1}^{2k}(\vec{Z}_{2k})\right)\Big] \frac{\partial}{\partial z_{3k+3}} \end{aligned} \end{equation} for $ x_1,x_2,x_3 \in \Psi_{2k+1}$. The field $\gamma^{2k+1}$ is the same as in the case $a_3\neq 0$.

For the case $ K=2k\ge 3 $ even the analogous procedure leads to the following complete fields on  ($\mathbb{C}^3)^K $ tangent to the fibers of $ \pi_4\circ \Psi_K $:
\begin{equation} \label{e:thetaevenfield}
\begin{aligned}\theta_{x_1x_2x_3}^{2k}&=P_{1}^{2k-1}(\vec{Z}_{2k-1})^2\partial_{x_1x_2x_3}^{2k-1}+\\&+P_{1}^{2k-1}(\vec{Z}_{2k-1})\partial_{x_1x_2x_3}^{2k-1}\left(P_{4}^{2k-1}(\vec{Z}_{2k-1})\right)\frac{\partial}{\partial w_{3k-1}}+\\ &+\Big[P_{1}^{2k-1}(\vec{Z}_{2k-1})\partial_{x_1x_2x_3}^{2k-1}\left(P_{3}^{2k-1}(\vec{Z}_{2k-1})\right)-\\&-P_{2}^{2k-1}(\vec{Z}_{2k-1})\partial_{x_1x_2x_3}^{2k-1}\left(P_{4}^{2k-1}(\vec{Z}_{2k-1})\right)\Big] \frac{\partial}{\partial w_{3k-2}} \end{aligned}\end{equation}
for $x_1,x_2,x_3 \in \Xi_{2k}$,
\begin{equation} \label{e:phievenfield}
\begin{aligned}\phi_{x_1x_2x_3}^{2k}&=P_2^{2k-1}(\vec{Z}_{2k-1})^2\partial_{x_1x_2x_3}^{2k-1}+\\&+P_2^{2k-1}(\vec{Z}_{2k-1})\partial_{x_1x_2x_3}^{2k-1}\left(P_{3}^{2k-1}(\vec{Z}_{2k-1})\right)\frac{\partial}{\partial w_{3k-1}}+\\ &+\Big[P_2^{2k-1}(\vec{Z}_{2k-1})\partial_{x_1x_2x_3}^{2k-1}\left(P_{4}^{2k-1}(\vec{Z}_{2k-1})\right)- \\ & -P_{1}^{2k-1}(\vec{Z}_{2k-1})\partial_{x_1x_2x_3}^{2k-1}\left(P_{3}^{2k-1}(\vec{Z}_{2k-1})\right)\Big] \frac{\partial}{\partial w_{3k}} \end{aligned} 
\end{equation}
for $x_1,x_2,x_3 \in \Xi_{2k}$ and
\begin{equation} \label{e:gammaeven} \begin{aligned}
\gamma^{2k}=P_1^{2k-1}(\vec{Z}_{2k-1})^2\frac{\partial}{\partial w_{3k}}&+P_{2}^{2k-1}(\vec{Z}_{2k-1})^2\frac{\partial}{\partial w_{3k-2}}- \\ &-P_1^{2k-1}(\vec{Z}_{2k-1})P_{2}^{2k-1}(\vec{Z}_{2k-1})\frac{\partial}{\partial w_{3k-1}}. \end{aligned}
\end{equation} 

\begin{Rem} \label{r:levelFields}
It follows from the inductive formulas (\ref{e:oddcase}) and  (\ref{e:evencase}) that $ \theta_{x_1x_2x_3}^{K} $, $ \phi_{x_1x_2x_3}^{K} $ and $ \gamma^K $, considered as vector fields on $ (\mathbb{C}^3)^L $, are tangent to the fibers $ \mathcal{F}_{(a_1,a_2,a_3a_4)}^L $ for $ L\ge K $. In other words, the fields associated with triples introduced on level $K$ are tangential to all fibers $\mathcal{F}^L$ for $L \ge K$.
\end{Rem}

\section{Strategy of proof of stratified ellipticity}\label{s:strategy}

We outline the strategy for proving that the submersion is a stratified elliptic submersion. We have seen that the fibers are given by four polynomial equations. We have also seen that these four equations can be reduced to two equations. We then use the exact form of these two equations to find $\Xi_K$ so that $\partial_{x_1x_2x_3}^K$ are complete vector fields exactly when $x_1,x_2,x_3 \in \Xi_{K}$. This leads us to the globally defined complete vector fields $\theta_{x_1x_2x_3}^K$, $\phi_{x_1x_2x_3}^K$ and $\gamma^K$ described in Section \ref{s:descComplete}. Find a big (a complement of an analytic subset) "good" set on the fibers where the collection of these vector fields spans the tangent space of the fiber. For points outside the good set find a complete field $V$ whose orbit through the point intersects the good set. At points along the orbit that are also in the good set the collection of complete vector fields above spans. Now pull back the collection of vector fields by suitable flow automorphisms of $V$ and add these fields to the collection (see Definition \ref{d:generated}). This enlarged collection of complete vector fields spans in a bigger set thus enlarging the good set. Continue this enlarging of the collection of vector fields until it spans the tangent space at every point of every fiber in the stratum. To accomplish this strategy we need the following technical results.

\begin{Lem} \label{l:finitecollection}
Let $M$ be a Stein manifold, $N_0 \subset N \subset M$ analytic subvarieties. Given a finite collection $\theta_1, \ldots, \theta_k$ of complete holomorphic vector fields on $M$ which span the tangent space $T_x M$ at all points
$x\in M\setminus N$ and given another  complete holomorphic vector field $\phi$ on $M$ (whose flow we denote by $\alpha_t \in \operatorname{Aut}_{hol} (M)$, $t\in \mathbb{C}$) with the property that the orbit through points of $N\setminus N_0$
is leaving $N$, i.e.  $\{ \alpha_t (x) : t \in \mathbb{C} \} \not\subset N$ $\forall x \in N\setminus N_0$. Then there are finitely many times $t_i \in \mathbb{C}$ $i=1,\ldots, l$ such that $\{ \alpha_{t_i} (x) \}_{i=1}^l \not\subset N$ $\forall x\in N\setminus N_0$. In particular the finite collection $ \{ \alpha_{t_i}^\star (\theta_m) \}_{i=1, m=1}^{l,k} $ of complete holomorphic vector fields on $M$ is spanning the tangent space $T_x M$ at all points
$x\in M\setminus N_0$.
\end{Lem}   

\begin{proof} The analytic subset $N$ has at most countably many components. Denote by  $B_i$ those components which are not entirely contained in $N_0$.
Define $a_0$ to be the maximal dimension of them.
Choose a point $x_i$ from each of those $B_i$. 
 For every $i$ the set $A_i:=  \{ t\in \mathbb{C} : \alpha_t  (x_i) \in N\}$ is  discrete. 
Since a countable union of discrete sets is meagre in $\mathbb{C}$, we find $t_1 \notin A_i \ \forall i$.
Denote by  $\tilde B_i$ those components of the  analytic subset $N_1 := \{ y\in N :  \alpha_{t_1} (y) \in N \}$ which are not entirely contained in $N_0$ and define $a_1$ to be the maximal dimension of them.
By construction $a_1 < a_0$. 
Choose a point $\tilde x_i$ from each of those $\tilde B_i$. 
 For every $i$ the set  $\tilde A_i:= \{ t\in \mathbb{C} : \alpha_t  (\tilde x_i) \in N\} $ is  discrete. 
Since a countable union of discrete sets is meagre in $\mathbb{C}$, we find $t_2 \notin \tilde A_i  \ \forall i$.

 Let $a_2$ be the maximal dimension of those components of the  analytic subset $N_2 := \{ y\in N :  \alpha_{t_1} (y) \in N \  \text{and} \ \alpha_{t_2} (y) \in N \}$ which are not entirely contained in $N_0$. By construction 
 $a_2 < a_1$ and continuing the construction after finitely steps we reach our conclusion.
\end{proof}
The next Lemma is a generalized and parametrized version of the previous one. It is adapted to 
the stratified spray situation. Namely, we have to produce sprays 
not on a single fiber but in a neighborhood of the fiber in each stratum (see Definition \ref{definstratifiedspray}). In fact in our case it will be on the whole stratum. The following definitions are straightforward.

\begin{Def} Let $\pi: X \to Y$ is a holomorphic map between complex manifolds and denote ${\rm d}\pi : TX \to TY$ the tangent map. We call a holomorphic vector field  $\theta$ on $X$ {\sl fiber preserving} if $d\pi (\theta) =0$. 
\end{Def}
\begin{Def} A subset $N$ of a complex manifold $M$
is called invariant with respect to a collection 
of vector fields on $M$ if for each of the vector fields holds: For each starting point 
$x\in N$ the local flow of the field (which is defined in a neighborhood of $0$) remains contained in $N$.
\end{Def}

\begin{Lem} Given a submersion $\pi: M \to Y $ with connected fibers $M_y:= \pi^{-1} (y)$ and given a finite collection  of complete fiber preserving holomorphic vector fields on $M$ such that in each fiber $M_y$ there is a point $x\in M_y$ where they span the tangent space $T_x M_y$. Moreover there is no analytic subset $N$ of $M$ contained in a fiber $M_y$ which is invariant under the flows of $\theta_1, \ldots, \theta_k$. Then a finite subset of the set
$\Gamma (\theta_1, \ldots, \theta_k)$ is spanning $T_x M_{\pi (x)}$ for all $x\in M$.
\end{Lem}
\begin{proof} 
Let $N\subset M$ be the set of points $x$ where $span \{(\theta_1, \ldots, \theta_k) \} \ne T_x M_{\pi (x)}$. 
by assumption $N \cup M_y$ is a proper analytic subset of $M_y$ for each $y \in Y$. Since there is no invariant 
analytic subset different from the fibers for each $x_0 \in N$ there is a field $\theta_i$ whose flow starting in
$x_0$ will leave $N$, i.e. go through points where $(\theta_1, \ldots, \theta_k)$ span $T_x M_{\pi (x)}$ Now choose (at most countably many) points, one from each component of $N$ and as in the proof of the proceeding lemma
find finitely many times $t_i$ such that enlarging the collection $\theta_1, \ldots, \theta_k$ by the pullbacks $(\alpha_i (t_i))^* (\theta_m)  i,m = 1, \ldots , k$ we get a new finite collection of complete fields where the set of points where this new collection does not span
the tangent space of the $\pi$-fiber has smaller dimension. By finite induction on the dimension we get the desired result.
\end{proof}

\section{Auxilliary quantities and results} \label{s:helpful}

Define \begin{equation} \renewcommand*{\arraystretch}{1.5} \mathcal{M}_{x_1x_2x_3}^K = \begin{pmatrix} \partial P_1^K/\partial x_1 & \partial P_1^K/\partial x_2 & \partial P_1^K/\partial x_3 \\ \partial P_2^K/\partial x_1 & \partial P_2^K/\partial x_2 & \partial P_2^K/\partial x_3 \\ \partial P_3^K/\partial x_1 & \partial P_3^K/\partial x_2 & \partial P_3^K/\partial x_3 \\ \partial P_4^K/\partial x_1 & \partial P_4^K/\partial x_2 & \partial P_4^K/\partial x_3\end{pmatrix} \end{equation} for any triple $ x_1,x_2,x_3 $ from $ \vec{Z}_K $. Removing the $j$-th row from $\mathcal{M}_{x_1x_2x_3}^{K}$ gives us  $(3\times 3)$-matrices which we denote by $\mathcal{M}_{x_1x_2x_3}^{K,j}$. Let $\mathcal{R}_{x_1x_2x_3}^{K,j} = \det \mathcal{M}_{x_1x_2x_3}^{K,j}$. The significance of the functions $\mathcal{R}_{x_1x_2x_3}^{K,j}$ is understood if one notices, because of (\ref{e:fields}), 
that \begin{equation}\label{e:ROdd1}  \mathcal{R}_{x_1x_2x_3}^{2k+1,1}=\partial_{x_1x_2x_3}^{2k}P_2^{2k},\end{equation} \begin{equation}\label{e:ROdd2} \mathcal{R}_{x_1x_2x_3}^{2k+1,2}=\partial_{x_1x_2x_3}^{2k}P_1^{2k} \end{equation} and that \begin{equation} \label{e:REven3} \mathcal{R}_{x_1x_2x_3}^{2k,3}=\partial_{x_1x_2x_3}^{2k-1}P_4^{2k-1},\end{equation} \begin{equation} \label{e:REven4} \mathcal{R}_{x_1x_2x_3}^{2k,4}=\partial_{x_1x_2x_3}^{2k-1}P_3^{2k-1}. \end{equation} From (\ref{e:oddcase}) and (\ref{e:evencase}) we get the relations \begin{equation}\label{e:indRoddtoeven}\renewcommand*{\arraystretch}{1.5} \begin{pmatrix} \mathcal{R}_{x_1x_2x_3}^{2k+1,1} \\ \mathcal{R}_{x_1x_2x_3}^{2k+1,2} \\ \mathcal{R}_{x_1x_2x_3}^{2k+1,3} \\ \mathcal{R}_{x_1x_2x_3}^{2k+1,4}\end{pmatrix} = \begin{pmatrix} 1 & 0 & 0 & 0 \\ 0 & 1 & 0 & 0 \\ -z_{3k+1} & z_{3k+2} & 1 & 0 \\ z_{3k+2} & -z_{3k+3} & 0 & 1 \end{pmatrix} \begin{pmatrix} \mathcal{R}_{x_1x_2x_3}^{2k,1} \\ \mathcal{R}_{x_1x_2x_3}^{2k,2} \\ \mathcal{R}_{x_1x_2x_3}^{2k,3} \\ \mathcal{R}_{x_1x_2x_3}^{2k,4}\end{pmatrix} \end{equation} and \begin{equation} \label{e:indReventoodd}\renewcommand*{\arraystretch}{1.5} \begin{pmatrix} \mathcal{R}_{x_1x_2x_3}^{2k,1} \\ \mathcal{R}_{x_1x_2x_3}^{2k,2} \\ \mathcal{R}_{x_1x_2x_3}^{2k,3} \\ \mathcal{R}_{x_1x_2x_3}^{2k,4}\end{pmatrix} = \begin{pmatrix} 1 & 0 & -w_{3k-2} & w_{3k-1} \\ 0 & 1 & w_{3k-1} & -w_{3k} \\ 0 & 0 & 1 & 0 \\ 0 & 0 & 0 & 1 \end{pmatrix} \begin{pmatrix} \mathcal{R}_{x_1x_2x_3}^{2k-1,1} \\ \mathcal{R}_{x_1x_2x_3}^{2k-1,2} \\ \mathcal{R}_{x_1x_2x_3}^{2k-1,3} \\ \mathcal{R}_{x_1x_2x_3}^{2k-1,4}\end{pmatrix} .\end{equation}

Consider the vector fields $\theta_{x_1x_2x_3}^L$ and $\phi_{x_1x_2x_3}^L$ where $(x_1,x_2,x_3)\in \mathcal{T}_L$ and $3 \le L \le K$. Rewriting (\ref{e:thetafieldodd}), (\ref{e:phifieldodd}), (\ref{e:thetaevenfield}) and (\ref{e:phievenfield}) using these functions we get \begin{equation}\label{e:thetafieldoddformulasWithR} \begin{aligned}\theta_{x_1x_2x_3}^{2k+1}&=P_3^{2k}(\vec{Z}_{2k})^2\partial_{x_1x_2x_3}^{2k}+P_3^{2k}(\vec{Z}_{2k})\mathcal{R}_{x_1x_2x_3}^{2k+1,1}(\vec{Z}_{2k})\frac{\partial}{\partial z_{3k+2}}+\\ &+\left(P_3^{2k}(\vec{Z}_{2k})\mathcal{R}_{x_1x_2x_3}^{2k+1,2}(\vec{Z}_{2k})-P_{4}^{2k}(\vec{Z}_{2k})\mathcal{R}_{x_1x_2x_3}^{2k+1,1}(\vec{Z}_{2k})\right) \frac{\partial}{\partial z_{3k+1}},  
\end{aligned}
\end{equation}
\begin{equation}\label{e:phifieldoddformulasWithR}
\begin{aligned}
\phi_{x_1x_2x_3}^{2k+1}&=P_4^{2k}(\vec{Z}_{2k})^2\partial_{x_1x_2x_3}^{2k}+P_4^{2k}(\vec{Z}_{2k})\mathcal{R}_{x_1x_2x_3}^{2k+1,2}(\vec{Z}_{2k})\frac{\partial}{\partial z_{3k+2}}+\\ &+\left(P_4^{2k}(\vec{Z}_{2k})\mathcal{R}_{x_1x_2x_3}^{2k+1,1}(\vec{Z}_{2k})-P_{3}^{2k}(\vec{Z}_{2k})\mathcal{R}_{x_1x_2x_3}^{2k+1,2}(\vec{Z}_{2k})\right) \frac{\partial}{\partial z_{3k+3}} ,
\end{aligned}
\end{equation}
\begin{equation}\label{e:thetafieldevenformulasWithR}
\begin{aligned}\theta_{x_1x_2x_3}^{2k}&=P_{1}^{2k-1}(\vec{Z}_{2k-1})^2\partial_{x_1x_2x_3}^{2k-1}+P_{1}^{2k-1}(\vec{Z}_{2k-1})\mathcal{R}_{x_1x_2x_3}^{2k,3}(\vec{Z}_{2k-1})\frac{\partial}{\partial w_{3k-1}}+\\ &+\left(P_{1}^{2k-1}(\vec{Z}_{2k-1})\mathcal{R}_{x_1x_2x_3}^{2k,4}(\vec{Z}_{2k-1})-P_{2}^{2k-1}(\vec{Z}_{2k-1})\mathcal{R}_{x_1x_2x_3}^{2k,3}(\vec{Z}_{2k-1})\right) \frac{\partial}{\partial w_{3k-2}} 
\end{aligned}
\end{equation}
and 
\begin{equation}\label{e:phifieldevenformulasWithR}
\begin{aligned}
\phi_{x_1x_2x_3}^{2k}&=P_2^{2k-1}(\vec{Z}_{2k-1})^2\partial_{x_1x_2x_3}^{2k-1}+P_2^{2k-1}(\vec{Z}_{2k-1})\mathcal{R}_{x_1x_2x_3}^{2k,4}(\vec{Z}_{2k-1})\frac{\partial}{\partial w_{3k-1}}+\\ &+\left(P_2^{2k-1}(\vec{Z}_{2k-1})\mathcal{R}_{x_1x_2x_3}^{2k,3}(\vec{Z}_{2k-1})-P_{1}^{2k-1}(\vec{Z}_{2k-1})\mathcal{R}_{x_1x_2x_3}^{2k,4}(\vec{Z}_{2k-1})\right) \frac{\partial}{\partial w_{3k}} \end{aligned} \end{equation} We see that half of the functions $\mathcal{R}_{x_1x_2x_3}^{K,j}$ occur in the coefficients of the last three directions. As already observed the fields $\theta_{x_1x_2x_3}^L$ and $\phi_{x_1x_2x_3}^L$ for $L<K$ have   components zero along the last three directions. We have to make sure that the projection onto the last three variables of the collection of fields $\theta_{x_1x_2x_3}^K$ and $\phi_{x_1x_2x_3}^K$ spans a three-dimensional space. If this is true for a point we will say that the fields span all new direction in the point. In order to determine if our fields span all new directions in a point $\vec{Z}_K\in \mathcal{F}_{a_1a_2a_3a_4}^K$ we will use the following. Let $N_K=|\mathcal{T}_K|$ be the number of complete triples. Define the $(2\times N_{K-1})$-matrices \[  \Omega_{x_1x_2x_3}^K(\vec{Z}_K) = \begin{cases} \renewcommand*{\arraystretch}{1.5}\begin{pmatrix} \mathcal{R}_{x_1x_2x_3}^{K-1,1}(\vec{Z}_{K}) & \cdots \\ \mathcal{R}_{x_1x_2x_3}^{K-1,2}(\vec{Z}_{K}) & \cdots  \end{pmatrix} \text{ when $K$ odd,}\\ \\  \renewcommand*{\arraystretch}{1.5}\begin{pmatrix} \mathcal{R}_{x_1x_2x_3}^{K-1,3}(\vec{Z}_{K}) & \cdots \\ \mathcal{R}_{x_1x_2x_3}^{K-1,4}(\vec{Z}_{K}) & \cdots  \end{pmatrix} \text{ when $K$ even,}\end{cases}  
\] where $(x_1,x_2,x_3)$ run over all triples in $\mathcal{T}_{K-1}$. Using the formulas (\ref{e:thetafieldoddformulasWithR}), (\ref{e:phifieldoddformulasWithR}), (\ref{e:thetafieldevenformulasWithR}), (\ref{e:phifieldevenformulasWithR}), and remembering that a fiber $\mathcal{F}_{(a_1,a_2,a_3,a_4)}^K$ is called generic if $(a_1,a_2)\neq (0,0)$ when $K$ is even and if $(a_3,a_4)\neq (0,0)$ when $K$ is odd  it is an exercise in linear algebra to prove the lemma below.

\begin{Lem} \label{l:spanNewDirections}
If in a point $\vec{Z}_K\in \mathcal{F}_{a_1a_2a_3a_4}^K$ in a generic fiber \[ \operatorname{Rank}\Omega_{x_1x_2x_3}^K(\vec{Z}_K)=2 \] then \[ \left\{\theta_{x_1x_2x_3}^{K};(x_1,x_2,x_3)\in \mathcal{T}_{K-1} \right\}\cup \left\{\phi_{x_1x_2x_3}^{K};(x_1,x_2,x_3)\in \mathcal{T}_{K-1} \right\} \cup \left\{\gamma^{K}\right\} \] span all three new directions. If  \[ \operatorname{Rank}\Omega_{x_1x_2x_3}^K(\vec{Z}_K)=1 \] then \[ \left\{\theta_{x_1x_2x_3}^{K};(x_1,x_2,x_3)\in \mathcal{T}_{K-1} \right\}\cup \left\{\phi_{x_1x_2x_3}^{K};(x_1,x_2,x_3)\in \mathcal{T}_{K-1} \right\} \cup \left\{\gamma^{K}\right\} \] span two out of three new directions. 
\end{Lem}  

Because of the formulas (\ref{e:indRoddtoeven}) and (\ref{e:indReventoodd}) we have the lemma below.

\begin{Lem} \label{l:rankIsStable}
Let $K\le L$ and put \[\mathcal{M}_K^{L}(\vec{Z}_L)= \renewcommand*{\arraystretch}{1.5}\begin{pmatrix} \mathcal{R}_{x_1x_2x_3}^{L,1}(\vec{Z}_L) & \dots \\ \mathcal{R}_{x_1x_2x_3}^{L,2}(\vec{Z}_L) & \dots \\ \mathcal{R}_{x_1x_2x_3}^{L,3}(\vec{Z}_L) & \dots \\ \mathcal{R}_{x_1x_2x_3}^{L,4}(\vec{Z}_L) & \dots \end{pmatrix} \] where $(x_1,x_2,x_3)$ run over all triples in $\mathcal{T}_K$. For all $L \ge K$ \[\operatorname{Rank}\mathcal{M}_K^K(\vec{Z}_L)=\operatorname{Rank}\mathcal{M}_K^L(\vec{Z}_L).\]
\end{Lem}

The importance of Lemma \ref{l:spanNewDirections} and Lemma \ref{l:rankIsStable} is seen in the following corollary.

\begin{Cor}\label{c:newdirections}
Let $L>K$ and $\vec{Z}_K$ be a point where $\operatorname{Rank} \mathcal{M}_K^K(\vec{Z}_K)=4$. Then for all points $\vec{Z}_L$ contained in a generic fiber $\mathcal{F}_{(a_1,a_2,a_3,a_4)}^L$ such that $\pi(\vec{Z}_L)=\vec{Z}_K$ the complete fields \[ \left\{\theta_{x_1x_2x_3}^{L};(x_1,x_2,x_3)\in \mathcal{T}_{L-1} \right\}\cup \left\{\phi_{x_1x_2x_3}^{L};(x_1,x_2,x_3)\in \mathcal{T}_{L-1} \right\} \cup \left\{\gamma^{L}\right\} \] span all new directions (the directions along the last three variables in $ (\mathbb{C}^3)^L$). 	
\end{Cor}
\begin{proof}
	Two rows of the rank 4 matrix $\mathcal{M}_K^L(\vec{Z}_L)$ are linearly independent. 
\end{proof}

\begin{table} 
\renewcommand*{\arraystretch}{1.5}
\begin{tabulary}{0.9\textwidth}{|C|C|C|C|C|}
\hline
 $\partial_{x_1x_2x_3}^2$ & $\mathcal{R}_{x_1x_2x_3}^{2,1}$ & $\mathcal{R}_{x_1x_2x_3}^{2,2}$ & $\mathcal{R}_{x_1x_2x_3}^{2,3}$ & $\mathcal{R}_{x_1x_2x_3}^{2,4}$ \\
\hline
$ \partial_{w_1w_2w_3}^2 $ & 0 & 0 & 0 & 0 \\
\hline
$ \partial_{z_2w_2w_3}^2 $ & 0 & $z_3^2$ & 0 & 0 \\
\hline
$ \partial_{z_3w_1w_2}^2 $ & $z_2^2$ & 0 & 0 & 0 \\
\hline
$ \partial_{z_2w_1w_3}^2 $ & 0 & $z_2z_3$ & 0 & 0 \\
\hline
$ \partial_{z_3w_1w_3}^2 $ & $z_2z_3$ & 0 & 0 & 0 \\
\hline
$ \partial_{z_2z_3w_1}^2 $ & $z_2w_2$ & $-z_2w_3$ & 0 & $z_2$ \\
\hline
$ \partial_{z_2z_3w_3}^2 $ & $-z_3w_1$ & $z_3w_2$ & $z_3$ & 0 \\
\hline
\end{tabulary}
\caption{The expressions for $\mathcal{R}_{x_1x_2x_3}^{2,i}$.} \label{table2}
\end{table}

\begin{Cor} \label{c:newdir}
	Let $L\ge 3$ and $\vec{Z}_L$ be a point that is contained in a generic fiber $\mathcal{F}_{(a_1,a_2,a_3,a_4)}^L$ and such that $z_2z_3\neq 0$. Then \[ \left\{\theta_{x_1x_2x_3}^{L};(x_1,x_2,x_3)\in \mathcal{T}_{L-1} \right\}\cup \left\{\phi_{x_1x_2x_3}^{L};(x_1,x_2,x_3)\in \mathcal{T}_{L-1} \right\} \cup \left\{\gamma^{L}\right\} \] span all new directions in $\vec{Z}_L$.
\end{Cor} 
In order to use this corollary we need the following lemma.
\begin{Lem} \label{l:z2z3notzero} We have the following cases for the function $P=z_2z_3$ and the fibers $\mathcal{F}^K$:
\begin{enumerate}
    \item  $P$ is not identically zero on $\mathcal{F}_{(a_1,a_2,a_3,a_4)}^K$ for $K\ge 5$. For these $K$ the fibers $\mathcal{F}_{(a_1,a_2,a_3,a_4)}^K$ are irreducible.
    \item The fibers $\mathcal{F}_{(a_1,a_2,a_3,a_4)}^4$ are irreducible except when $$(a_1,a_2,a_3,a_4)=(0,0,0,1).$$ The function $P$ is not identically zero on fibers except for one component of $\mathcal{F}_{(0,0,0,1)}^4$.
    \item The fibers $\mathcal{F}_{(a_1,a_2,a_3,a_4)}^3$ are irreducible except when $$(a_1,a_2,a_3,a_4)=(a_1,a_2,0,1).$$ 
    The function $P$ is not identically zero on fibers except on $\mathcal{F}_{(0,a_2,0,0)}^3$ or $\mathcal{F}_{(a_1,0,0,0)}^3$ or on one component of the reducible fiber $\mathcal{F}_{(a_1,a_2,0,1)}^3$ where it is identically zero.
\end{enumerate}
       
\end{Lem} 

\begin{proof}
We first prove (3). The fibers $\mathcal{F}_{(a_1,a_2,0,0)}^3$ are just biholomorphic to $\mathbb{C}^5$ and $Z_2, z_3$ are constantly equal to $a_1, a_2$. This shows they are irreducible
and the assertion about the function $P$. The fibers $\mathcal{F}_{(a_1,a_2,0,1)}^3$ are
isomorphic to the variety  $ \mathcal{G}_{(0,1)}^2 $ given by two equations which can be written in matrix form as \begin{equation} 
\begin{pmatrix} w_1 & w_2 \\ w_2 & w_3 \end{pmatrix}\begin{pmatrix} z_2 \\ z_3 \end{pmatrix} = \begin{pmatrix} 0 \\ 0 \end{pmatrix}.
\end{equation} 
From this it can be seen that $ \mathcal{G}_{(0,1)}^2 $ has two irreducible components. One is \begin{equation} \label{e:a1}
A_1=\{z_2=z_3=0\}\cong \mathbb{C}_{w_1w_2w_3}^3 	
\end{equation}  and the other is  \begin{equation} \label{e:a2} A_2=\left\{\begin{pmatrix} w_1 & w_2 \\ w_2 & w_3 \end{pmatrix}\begin{pmatrix} z_2 \\ z_3 \end{pmatrix} = \begin{pmatrix} 0 \\ 0 \end{pmatrix} \text{ and } \det \begin{pmatrix} w_1 & w_2 \\ w_2 & w_3 \end{pmatrix}=0\right\}.\end{equation} The singularity set of $ \mathcal{G}_{(0,1)}^2 $ is $ A_1 \cap A_2 $. Clearly $P$ is identically zero on $A_1$ and not identically zero on
$A_2$. Observe that $\mathcal{F}_{(a_1,a_2,0,1)}^3$ are connected, their smooth part consists of the two connected
components $A_1 \setminus A_2 $ and $A_2 \setminus A_1$.

The smooth generic fibers $\mathcal{F}_{(a_1,a_2,a_3,a_4)}^3$ for $(a_3, a_4) \notin \{(0,0), (0,1)\}$ are isomorphic to the variety  $ \mathcal{G}_{(a_3,a_4)}^2 $ given by
the two equations
\begin{equation} z_2 w_1 + z_3 w_2 = a_3
\end{equation}
and 
\begin{equation} z_2 w_2 + z_3 w_3 + 1 = a_4
\end{equation}
In case $z_2 \ne 0$ these equations can be used to express $w_2$ and $w_3$ by the other variables and we get a chart
isomorphic to $\mathbb{C}_{z_2}^\star  \times \mathbb{C}_{z_3}\times \mathbb{C}_{w_3}$. In case $z_3 \ne 0$ we can express $w_2$ and $w_3$ which gives us a similar chart. Thus  $ \mathcal{G}_{(a_3,a_4)}^2 $ is covered by two connected charts with non-empty intersection which shows that it is connected.  Thus the smooth generic fibers $\mathcal{F}_{(a_1,a_2,a_3,a_4)}^3$ are irreducible. The function $P$ is not identically zero on both charts. The assertion (3) is completely proven.

Next we prove assertion (2). The non-generic fibers $\mathcal{F}_{(0,0,a_3,a_4)}^4$ are isomorphic to
$\mathcal{F}_{(0,0,a_3,a_4)}^3 \times \mathbb{C}^3$, where $\mathbb{C}^3$ corresponds to the new variables $w_4, w_5, w_6$. All assumptions about these fibers follow therefore from the corresponding assumptions about 
$\mathcal{F}_{(0,0,a_3,a_4)}^3$.  

In the case of generic fibers which are known to be smooth (see Section \ref{s:stratification}) we just have to prove they are connected. 
For this consider \begin{equation}
 	\mathcal{F}_{(a_1,a_2,a_3,a_4)}^4=\bigcup_{(w_4,w_5,w_6)\in \mathbb{C}^3} \mathcal{F}_{(a_1,a_2,b_3,b_4)}^3 \end{equation} where $b_3= a_3-w_4a_1-w_5a_2$ and $b_4= a_4-w_5a_1-w_6a_2$.
 	In other words we consider  the surjective projection $\rho: \mathcal{F}_{(a_1,a_2,a_3,a_4)}^4\to \mathbb{C}^3$, mapping a point to
 	its last three coordinates $(w_4, w_5, w_6)$ where the $\rho$-fibers are just fibers $\mathcal{F}_{(a_1,a_2,b_3,b_4)}^3 $. Connectedness of the $\rho$-fibers implies that  a connected component of 
 	$\mathcal{F}_{(a_1,a_2,a_3,a_4)}^4$ has to be $\rho$-saturated.
 	Since $\rho$ is a submersion in generic points of the fiber (it is not a submersion only in singular points of an  $\mathcal{F}_{}^3 $-fiber) any connected components of $\mathcal{F}_{(a_1,a_2,a_3,a_4)}^4$ is equal to $\rho^{-1} (U)$, where $U$ is some open subset of the base $\mathbb{C}^3$. Since 
 	the base is connected and $\rho$ is surjective connectedness of 
 	$\mathcal{F}_{(a_1,a_2,a_3,a_4)}^4$ follows. The function $P$ is
 	not identically on any $\mathcal{F}_{}^3 $-fiber contained in
 	$\mathcal{F}_{(a_1,a_2,a_3,a_4)}^4$, thus not identically zero
 	on $\mathcal{F}_{(a_1,a_2,a_3,a_4)}^4$ itself. This concludes the
 	proof of (2).
 	
 Last we prove assertion (1). The connectedness of 	the fibers 
 $\mathcal{F}_{(a_1,a_2,a_3,a_4)}^K$ for $K\ge 5$ can be proven by induction in a similar way as the connectedness of the generic 
 $\mathcal{F}_{}^4$-fibers is deduced from the properties of 
 $\mathcal{F}_{}^3$-fibers. We consider as above the surjective projection $\rho : \mathcal{F}_{(a_1,a_2,a_3,a_4)}^K \to \mathbb{C}^3$ onto the last three variables who's fibers are 
 $\mathcal{F}_{}^{K-1} $-fibers. Since again $\mathcal{F}_{}^{K-1} $-fibers are connected and  $\rho$ is a submersion in  smooth points of the 
 $\mathcal{F}_{}^{K-1} $-fibers, any connected component of $\mathcal{F}_{(a_1,a_2,a_3,a_4)}^K$ is 
 of the form $\rho^{-1} (U)$, where $U$ is some open subset of the base $\mathbb{C}^3$. 
 
 In addition we will prove by induction that the smooth part
  of the singular
 fibers $\mathcal{F}_{(a_1,a_2,a_3,a_4)}^K\setminus \operatorname{Sing} \mathcal{F}_{(a_1,a_2,a_3,a_4)}^K $ is connected for $K\ge 5$. Together with connectedness of the fibers  this implies  the irreducibility of the fibers. 
 
 For even $K$ the singular fibers are the singular $\mathcal{F}_{}^{K-1} $-fibers times
 $\mathbb{C}^3$ and therefore the connectedness of the smooth part follows by induction hypothesis. 
 
 For odd $K=2k+1$ we are faced with the following situation: 
 The singular fibre is $\mathcal{F}_{(a_1,a_2,0,1)}^K$ and it is fibered by $\mathcal{F}_{}^{K-1} $-fibers all of which are smooth except for the fibers $\mathcal{F}_{(0,0,0,1)}^{K-1}$. The union of those fibers forms a codimension $2$ subvariety of 
 $\mathcal{F}_{(a_1,a_2,0,1)}^K$ (given by the equations $z_{3k+2}=z_{3k+3}=0$). By the argument above the complement, call it W,
 of this union in  $\mathcal{F}_{(a_1,a_2,0,1)}^K$ is connected. The singular points
 of $\mathcal{F}_{(a_1,a_2,0,1)}^K$ are contained in that union and is contained in (but not equal to) the union of the singular points of the fibers $\mathcal{F}_{(0,0,0,1)}^{K-1}$. We want to prove 
 that any smooth point $p$ of $\mathcal{F}_{(a_1,a_2,0,1)}^K$ which is contained in a fiber $\mathcal{F}_{(0,0,0,1)}^{K-1}$ is contained in the in the connected component containing $W$. Since the complement of $W$
  has codimension 2 in $\mathcal{F}_{(a_1,a_2,0,1)}^K$ an open neighborhood of $p$ in $\mathcal{F}_{(a_1,a_2,0,1)}^K$ has to intersect $W$, which gives the desired conclusion.
  
  As in the proof of (2) the function $P$ cannot be identical zero on any fiber  $\mathcal{F}_{(a_1,a_2,a_3,a_4)}^K$
  since this fiber contains $\mathcal{F}_{}^{K-1} $-fibers on which by induction hypothesis $P$ is not
  identical zero.

\end{proof}

\begin{Rem} The fact that after a certain number of factors the fibers of the fibration become all irreducible is very general. It was proven by J. Draisma as an outcome of an interesting discussion with the second author. The irreducubility statement in our lemma is just an example  of  a much more general property. We refer the interested reader to \cite{Draisma}. The exact number from which on irreducibility of the fibers  holds (in our case $5$) is not known in general, although Draisma gives a bound.
\end{Rem}

\begin{Def}\label{d:generated} Let $ M $ be a manifold and $ A $ be a set of complete vector fields on $ M $. The flows of elements of $ A $ give one-parameter subgroups of $ \operatorname{Aut}(M) $. Denote by $ S $ the group generated by elements of those one-parameter subgroups (finite compositions of time maps of vector fields of elements from $ A $). Define
\[ \Gamma(A) = \left\{ \alpha^*X, \text{ where }\alpha\in S \text{ and } X\in A \right\}.  \]
Obviously $ \Gamma(A) $ consists of complete vector fields and we call it {\it the collection generated by $ A $.}
\end{Def}

\begin{Def}\label{d:qL}
 	 Let $L\ge 3$. We define \[ \begin{aligned} & \mathcal{Q}_L  = \\ & =   \Gamma\left(\bigcup_{J=3}^L \left\{\left\{\theta_{x_1x_2x_3}^{J};(x_1,x_2,x_3)\in \Xi_{J} \right\}\cup \left\{\phi_{x_1x_2x_3}^{J};(x_1,x_2,x_3)\in \Xi_{J} \right\} \cup \left\{\gamma^{J}\right\}\right\}\right). \end{aligned} \]
 	\end{Def}
 	
 	At each step of the induction we will prove the following Proposition, which plays a crucial role in the inductive proof of Proposition
\ref{p:mainprop}.

\begin{Prop} \label{p:transitivity}
For each $L \ge 4$ holds: There are finitely many (complete) fields from  $ \mathcal{Q}_L $ which span the tangent space 
$T_x \mathcal{F}^L$ at each smooth point of any generic fiber $\mathcal{F}^L$. For $L=3$ there are finitely many (complete) fields from  $ \mathcal{Q}_3 $ which span the tangent space 
$T_x \mathcal{F}^3$ at each point of any smooth generic fiber $\mathcal{F}^3$.
\end{Prop}
\begin{Rem} \label{r:diffL3and4}
For $L=3$ singular generic fibers $\mathcal{F}^3_{(a_1,a_2,0,1)}$ have two irreducible components and we can prove the statement about smooth points on generic fibers only for one of those components. It is false for the other component.
\end{Rem}

\section{Proof of Proposition \ref{p:mainprop}: 3 matrix factors}\label{s:proof3factors}

In Table \ref{table1} we list the coefficients of the fields $ \partial_{x_1x_2x_3}^2 $ for all $ x_1, x_2, x_3 \in \mathcal{T}_2 $. 
\begin{table} 
\renewcommand*{\arraystretch}{1.5}
\begin{tabulary}{0.9\textwidth}{|C|C|C|C|C|C|}
\hline
  & $\partial / \partial z_2$ & $\partial / \partial z_3$ & $\partial / \partial w_1$ & $\partial / \partial w_2$ & $\partial / \partial w_3$ \\
\hline
$ \partial_{w_1w_2w_3}^2 $ & 0 & 0 & $ z_3^2 $ & $ -z_2z_3 $ & $ z_2^2 $ \\
\hline
$ \partial_{z_2w_2w_3}^2 $ & $ z_3^2 $ & 0 & 0 & $ -w_1z_3 $ & $w_1z_2-w_2z_3 $ \\
\hline
$ \partial_{z_3w_1w_2}^2 $ & 0 & $ z_2^2 $ & $ z_3w_3-w_2z_2 $ & $ -z_2w_3$ & 0 \\
\hline
$ \partial_{z_2w_1w_3}^2 $ & $z_2z_3$ & 0 & $-w_1z_3$ & 0 & $-z_2w_2$ \\
\hline
$ \partial_{z_3w_1w_3}^2 $ & 0 & $z_2z_3$ & $-w_2z_3$ & 0 & $-z_2w_3$ \\
\hline
$ \partial_{z_2z_3w_1}^2 $ & $-z_2w_3$ & $z_2w_2$ & $w_1w_3-w_2^2$ & 0 & 0 \\
\hline
$ \partial_{z_2z_3w_3}^2 $ & $w_2z_3$ & $w_1z_3$ & 0 & 0 &  $w_1w_3-w_2^2$ \\
\hline
\end{tabulary}
\caption{Coefficients of complete vector fields. For example, $\partial_{w_1w_2w_3}^2=z_3^2\frac{\partial}{\partial w_1}-z_2z_3\frac{\partial}{\partial w_2}+z_2^2\frac{\partial}{\partial w_3}$.} \label{table1}
\end{table}

We first consider the stratum of {\bf smooth generic fibers}, where we have $$ (a_3,a_4)\notin \{(0,0), (0,1)\}.$$ Notice that $ z_2=z_3=0 $ is contained in $ \mathcal{F}_{(z_5,z_6,0,1)}^3 $ and therefore $z_2$ and $z_3$ is never simultaneuously zero on any fiber in this stratum. It is enough to show that $ \mathcal{G}_{(a_3,a_4)}^2 $ is  elliptic. We see from the table that the fields $ \partial_{z_3w_1w_3}^2, \partial_{z_2w_1w_3}^2, \partial_{w_1w_2w_3}^2 $ span the tangent space $ T_{\vec{Z}_2}\mathcal{G}_{(a_3,a_4)}^2 $ for all points $ \vec{Z}_2 $ where $ z_2z_3\neq 0 $. The complement of this good set is the disjoint union of the analytic subsets $ \mathcal{A} =\left\{\vec{Z}_2; z_2=0 \right\} $ and $ \mathcal{B} =\left\{\vec{Z}_2; z_3=0 \right\}. $ From the table we see that $ \partial_{z_2w_2w_3}^2(z_2)=z_3^2 $ which is nowhere zero on $ \mathcal{A} $. Also $ \partial_{z_3w_1w_2}^2(z_3)=z_2^2 $ which is nowhere zero on $ \mathcal{B} $. By Lemma \ref{l:finitecollection} there exist finitely many complete fields from \[ \Gamma\left( \left \{\partial_{x_1x_2x_3}^2;(x_1,x_2,x_3)\in \mathcal{T}_2 \right \}\right) \] that span the tangent space $ T_{\vec{Z}_2}\mathcal{G}_{(a_3,a_4)}^2 $ for all points in the stratum. Therefore $\mathcal{G}_{(a_3,a_4)}^2 $ is elliptic. It follows that there are finitely many complete fields  from \[ \Gamma\left(\left\{\theta_{x_1x_2x_3}^3;(x_1,x_2,x_3)\in \mathcal{T}_2 \right\}\cup \left\{\phi_{x_1x_2x_3}^3;(x_1,x_2,x_3)\in \mathcal{T}_2 \right\}\cup \left \{\gamma^3\right \}\right) \] that span the tangent space $ T_{\vec{Z}_3}\mathcal{F}_{(a_1,a_2,a_3,a_4)}^3 $ for all points in the stratum. 

Now we consider the stratum of {\bf non-smooth generic fibers} $ (a_3,a_4)=(0,1) $. The two equations defining $ \mathcal{G}_{(0,1)}^2 $ can be written in matrix form as \begin{equation} 
\begin{pmatrix} w_1 & w_2 \\ w_2 & w_3 \end{pmatrix}\begin{pmatrix} z_2 \\ z_3 \end{pmatrix} = \begin{pmatrix} 0 \\ 0 \end{pmatrix}.
\end{equation} 
Recall that $ \mathcal{G}_{(0,1)}^2 $ has two irreducible components. The components are given by (see (\ref{e:a1}) and (\ref{e:a2})) \begin{equation}
A_1=\{z_2=z_3=0\}\cong \mathbb{C}_{w_1w_2w_3}^3 	
\end{equation}  and  \begin{equation} A_2=\left\{\begin{pmatrix} w_1 & w_2 \\ w_2 & w_3 \end{pmatrix}\begin{pmatrix} z_2 \\ z_3 \end{pmatrix} = \begin{pmatrix} 0 \\ 0 \end{pmatrix} \text{ and } \det \begin{pmatrix} w_1 & w_2 \\ w_2 & w_3 \end{pmatrix}=0\right\}.\end{equation} The singularity set of $ \mathcal{G}_{(0,1)}^2 $ is $ A_1\cap A_2 $. We have to show that the smooth part of $ \mathcal{G}_{(0,1)}^2 $, that is the disjoint union of $ A_1\setminus A_2 $ and  $ A_2\setminus A_1 $, is elliptic. In the proof for the smooth generic case it is shown that on the set where $z_2$ and $z_3$ are not both zero then there exists a collection of complete spanning vector fields. Since $ A_2\setminus A_1 $ is contained in that set we need only consider $ A_1\setminus A_2 $. The set $ A_1\setminus A_2 $ is biholomorphic to $ \mathbb{C}^3\setminus \{w_1w_3-w_2^2=0\} $.  The vector fields $(w_1w_3-w_2^2)\frac{\partial}{\partial w_1}$, $(w_1w_3-w_2^2)\frac{\partial}{\partial w_3}$, $2w_2\frac{\partial}{\partial w_1}+w_3\frac{\partial}{\partial w_2}$, $2w_2\frac{\partial}{\partial w_3}+w_1\frac{\partial}{\partial w_2}$   are complete on $ \mathbb{C}^3\setminus \{w_1w_3-w_2^2=0\} $ and span the tangent space in all points outside the analytic set $ A'=\{w_1=w_3=0\}\cap (A_1\setminus A_2) $. Since $w_2$ is nowhere zero on $A'$ any of the four complete fields points out of $A'$. By Lemma \ref{l:finitecollection} the proof is complete. Observe that we also have proved Propsition \ref{p:transitivity} for $L=3$. Notice that the fields $2w_2\frac{\partial}{\partial w_1}+w_3\frac{\partial}{\partial w_2}$, $2w_2\frac{\partial}{\partial w_3}+w_1\frac{\partial}{\partial w_2}$ are not in $\mathcal{Q}_3$ and this explains the difference between $L=3$ and $L\ge 4$ in Proposition \ref{p:transitivity}. See Remark \ref{r:diffL3and4}.

The stratum of {\bf non-generic fibers} is a locally trivial bundle with fibers $ \mathbb{C}^5(\cong \mathcal{F}_{(a_1,a_2,0,0)}^3) $ which is an elliptic submersion.

\section{Proof of Proposition \ref{p:mainprop}: 4 matrix factors}\label{s:proof4factors}

We begin the proof by studying the stratum of {\bf generic fibers}, $ (a_1,a_2)\neq (0,0) $. We write \begin{equation}\label{e:union4matrices}
 	\mathcal{F}_{(a_1,a_2,a_3,a_4)}^4=\bigcup_{(w_4,w_5,w_6)\in \mathbb{C}^3} \mathcal{F}_{(a_1,a_2,b_3,b_4)}^3 \end{equation} where $b_3= a_3-w_4a_1-w_5a_2$ and $b_4= a_4-w_5a_1-w_6a_2$. We need to find finitely many complete vector fields spanning $T_{\vec{Z}_4}\mathcal{F}_{(a_1,a_2,a_3,a_4)}^4$ for points $\vec{Z}_4$ in the stratum of generic fibers. Because of (\ref{e:union4matrices}) there are $b_3$ and $b_4$ so that $\vec{Z}_3\in \mathcal{F}_{(a_1,a_2,b_3,b_4)}^3$ and $\vec{Z}_4=(\vec{Z}_3,w_4,w_5,w_6)$.  We consider first the set of points in these fibers having the property that $ \mathbf{(b_3,b_4)\neq (0,0) \text{ or } (0,1)} $. Under these assumptions $\vec{Z}_4$ lies in a generic smooth fiber $\mathcal{F}_{(a_1,a_2,b_3,b_4)}^3$ and we know from Section \ref{s:proof3factors} that there is a finite collection of fields from $\mathcal{Q}_3$ which spans \[T_{\vec{Z}_4}\mathcal{F}_{(a_1,a_2,b_3,b_4)}^3\subset T_{\vec{Z}_4}\mathcal{F}_{(a_1,a_2,a_3,a_4)}^4.\] Corollary \ref{c:newdir} together with Lemma \ref{l:z2z3notzero}(3) shows that for the set defined by $z_2z_3\neq 0$ (which is a Zariski open and dense set of points of the generic fiber $\mathcal{F}_{(a_1,a_2,a_3,a_4)}^4$)  the fields  \[ \left\{\theta_{x_1x_2x_3}^{4};(x_1,x_2,x_3)\in \mathcal{T}_3 \right\}\cup \left\{\phi_{x_1x_2x_3}^{4};(x_1,x_2,x_3)\in \mathcal{T}_3 \right\} \cup \left\{\gamma^{4}\right\} \] span the new directions $w_4,w_5,w_6$. Since these new directions are complementary to \[T_{\vec{Z}_4}\mathcal{F}_{(a_1,a_2,b_3,b_4)}^3\subset T_{\vec{Z}_4}\mathcal{F}_{(a_1,a_2,a_3,a_4)}^4\] we have found finitely many complete fields spanning $T_{\vec{Z}_4}\mathcal{F}_{(a_1,a_2,a_3,a_4)}^4$ for points in a Zariski open dense set in all smooth generic fiber $\mathcal{F}_{(a_1,a_2,b_3,b_4)}^3$. Using Lemma \ref{l:finitecollection} we get finitely many complete fields spanning $T_{\vec{Z}_4}\mathcal{F}_{(a_1,a_2,a_3,a_4)}^4$ for all points in all generic fibers $\mathcal{F}_{(a_1,a_2,a_3,a_4)}^4$ with the property that $(b_3,b_4)\neq (0,0)$ or $(0,1)$.
 	Next we consider points $\vec{Z}_4$  where $ \mathbf{(b_3,b_4)=(0,1)} $, i.e, \[\vec{Z}_4\in \mathcal{F}_{(a_1,a_2,0,1)}^3 \subset \mathcal{F}_{(a_1,a_2,a_3,a_4)}^4.\] Remember that \begin{equation} \mathcal{F}_{(a_1,a_2,0,1)}^3=A_1\cup A_2 =A_1\dot{\cup} (A_2\setminus A_1)\end{equation} (see (\ref{e:a1}) and (\ref{e:a2})) where $A_1$ and $A_2$ are irreducible components. In the proof for $K=3$ we saw that there is a finite collection from $\mathcal{Q}_3$ which spans all tangent spaces \[T_{\vec{Z}_4}\mathcal{F}_{(a_1,a_2,0,1)}^3\subset T_{\vec{Z}_4}\mathcal{F}_{(a_1,a_2,a_3,a_4)}^4\] for all points in $A_2\setminus A_1$. Lemma \ref{l:z2z3notzero}(3) gives that $z_2z_3$ is not identically zero on $A_2\setminus A_1$ and as above appealing to Lemma \ref{l:finitecollection} we get spanning fields for the fiber $\mathcal{F}_{(a_1,a_2,a_3,a_4)}^4$ in all points of $A_2\setminus A_1$. Our aim is to exclude the existence of a subset of the fiber invariant under the flows of fields from $\mathcal{Q}_4$. By the reasoning above such a subset must be contained in $A_1$ or the set of points $\vec{Z}_4$ where $(b_3,b_4)=(0,0)$. Next we show that such a subset is disjoint from $A_1$. A calculation shows that \[\partial_{z_2z_3z_6}^3=(z_4w_2+z_5w_3)\frac{\partial}{\partial z_2}-(1+z_4w_1+z_5w_2)\frac{\partial}{\partial z_3}+\dots\] Therefore the complete fields \[\theta^4_{z_2z_3z_6}=a_1^2\partial^3_{z_2z_3z_6}+\dots\] and \[\phi^4_{z_2z_3z_6}=a_2^2\partial^3_{z_2z_3z_6}+\dots\] moves points out of $A_1$ (into the big orbit) unless in addition to $z_2=z_3=0$ also \begin{equation}\label{e:extracond}
 	    1+z_4w_1+z_5w_2=z_4w_2+z_5w_3=0.
 	\end{equation} Points in an invariant subset must satisfy also these equations. A calculation gives that $\partial^3_{z_4z_5z_6}=\frac{\partial}{\partial z_4}$ when $z_2=z_3=0$. Therefore the complete fields \[\theta^4_{z_4z_5z_6}=a_1^2\partial^3_{z_4z_5z_6}+\dots\]
 	and \[\phi^4_{z_4z_5z_6}=a_2^2\partial^3_{z_4z_5z_6}+\dots\] moves points out of this set since 
 	\[\theta^4_{z_4z_5z_6}(1+z_4w_1+z_5w_2)=a_1^2w_1, \]
 	\[\theta^4_{z_4z_5z_6}(z_4w_2+z_5w_3)=a_1^2w_2, \]
 	\[\phi^4_{z_4z_5z_6}(1+z_4w_1+z_5w_2)=a_2^2w_1, \]
 	\[\phi^4_{z_4z_5z_6}(z_4w_2+z_5w_3)=a_2^2w_2\] cannot all be zero, because this would contradict (\ref{e:extracond}). 
 	We now turn to points $\vec{Z}_4$ where $\mathbf{(b_3,b_4)=(0,0)}$ and again show that these points are not contained in an invariant subset and hence no such that invariant subset exists. We will find fields $\theta^4_{x_1x_2x_3}$ or $\phi^4_{x_1x_2x_3}$ such that $\mathcal{R}^{3,3}_{x_1x_2x_3}\neq 0$ or $\mathcal{R}^{3,4}_{x_1x_2x_3}\neq 0$. We begin by noticing that at points $\vec{Z}_4$ with $z_2z_3\neq 0$ we can leave the invariant set. Also $z_2=0= z_3$ cannot occur in $\mathcal{F}^3_{(a_1,a_2,0,0)}$. Two cases, $z_2\neq 0 = z_3$ and $z_2=0\neq z_3$, remains. Assume first that $\mathbf{z_2\neq 0 = z_3}$ (and $\mathbf{b_3=b_4=0}$). Here we begin by choosing the triple $(z_3,w_1,w_2)$. Since $\mathcal{R}^{3,3}_{z_3w_1w_2}=-z_2^2z_4$ and $\mathcal{R}^{3,4}_{z_3w_1w_2}=z_2^2z_5$ we move out of $\mathcal{F}^3_{(a_1,a_2,0,0)}$ unless $z_4=z_5=0$. Assuming in addition that $z_4=z_5=0$ we choose the triple $(z_2,z_3,w_1)$. For such points $\mathcal{R}^{3,4}_{z_2z_3w_1}=z_2+z_2w_3z_6$ (and  $\mathcal{R}^{3,3}_{z_2z_3w_1}=0$) so if $1+w_3z_6\neq 0$ we move out of $\mathcal{F}^3_{(a_1,a_2,0,0)}$. Choose $(z_2,z_3,z_6)$. Notice that  \[\theta^4_{z_2z_3z_6}=a_1^2\frac{\partial}{\partial z_6}\]
 	and \[\phi^4_{z_2z_3z_6}=a_2^2\frac{\partial}{\partial z_6}\] at these points and $\theta^4_{z_2z_3z_6}(1+w_3z_6)=a_1^2w_3$, $\phi^4_{z_2z_3z_6}(1+w_3z_6)=a_2^2w_3$ which both cannot be zero since $1+w_3z_6 = 0$ implies $w_3\neq 0$ and we assume that $(a_1,a_2)\neq (0,0)$.
 	
 	Now assume that $\mathbf{z_2 = 0 \neq z_3}$ (and also $\mathbf{b_3=b_4=0}$) and choose the triple $(z_2,w_2,w_3)$. Since $\mathcal{R}^{3,3}_{z_2w_2w_3}=z_3^2z_5$ and $\mathcal{R}^{3,4}_{z_2w_2w_3}=-z_3^2z_6$ we move out of $\mathcal{F}^3_{(a_1,a_2,0,0)}$ unless $z_5=z_6=0$. Assuming in addition that $z_5=z_6=0$ we choose the triple $(z_2,z_3,w_3)$. For such points $\mathcal{R}^{3,3}_{z_2z_3w_3}=z_3+z_3w_1z_4$ (and  $\mathcal{R}^{3,4}_{z_2z_3w_3}=0$) so if $1+w_1z_4\neq 0$ we move out of $\mathcal{F}^3_{(a_1,a_2,0,0)}$. Therefore assume also that $1+w_1z_4=0$ Choose $(z_2,z_3,z_4)$. Notice that  \[\theta^4_{z_2z_3z_4}=a_1^2\frac{\partial}{\partial z_4}\]
 	and \[\phi^4_{z_2z_3z_4}=a_2^2\frac{\partial}{\partial z_4}\] at these points and $\theta^4_{z_2z_3z_4}(1+w_1z_4)=a_1^2w_1$, $\phi^4_{z_2z_3z_4}(1+w_1z_4)=a_2^2w_1$ which both cannot be zero since $1+w_1z_4 = 0$ implies $w_1\neq 0$ and we assume that $(a_1,a_2)\neq (0,0)$.
 	This let us conclude that there is no invariant subset with respect to $\mathcal{Q}_4$ and we have handled the stratum of generic fibers. Note that this proves Proposition \ref{p:transitivity} for $K=4$.
 	
 	We need to study the stratum of {\bf non-generic fibers}. This stratum consists of those fibers where $\mathbf{a_1=a_2=0}$. We notice that these fibers satisfy \[ \mathcal{F}^4_{(0,0,a_3,a_4)}=\mathcal{F}^3_{(0,0,a_3,a_4)}\times \mathbb{C}^3 \] and since $\mathcal{F}^3_{(0,0,a_3,a_4)}$ is elliptic we have proven Proposition \ref{p:mainprop} for $K=4$.

\section{Proof of Proposition \ref{p:mainprop}: 5 matrix factors}\label{s:proof5factors}

We assume that $K=4$ and we have seen that the submersions $\Phi_L=\pi_4 \circ \Psi_L$ are stratified elliptic submersions when $3\le L \le 4$ and that Proposition \ref{p:transitivity} is true when $3\le L \le 4$. 

Study \begin{equation}\label{e:union5}
\mathcal{F}_{(a_1,a_2,a_3,a_4)}^{5}=\bigcup_{(z_{7},z_{8},z_{9})\in \mathbb{C}^3} \mathcal{F}_{(b_1,b_2,a_3,a_4)}^{4} 
\end{equation} 
where $b_1=a_1-z_{7}a_3-z_{8}a_4$ and $b_2=a_2-z_{8}a_3-z_{9}a_4$. Let $\vec{Z}_{5}\in \mathcal{F}_{(a_1,a_2,a_3,a_4)}^{5}$. Because of (\ref{e:union5}) there are $b_1$ and $b_2$ so that $\vec{Z}_{4}\in \mathcal{F}_{(b_1,b_2,a_3,a_4)}^{4}$ and \(\vec{Z}_{5}=(\vec{Z}_{4},z_{7},z_{8},z_{9}).\) 

First we study the stratum of {\bf smooth generic fibers}. Fibers in this stratum are those satisfying $\mathbf{(a_3,a_4)\not\in  \{(0,0),(0,1)\} }$. First notice that if $\mathbf{(b_1,b_2)\neq (0,0)}$ then \(\mathcal{F}_{(b_1,b_2,a_3,a_4)}^{4}\) is a generic smooth fiber for $\Phi_{4}$ and as above Proposition \ref{p:transitivity} for $L=4$, Corollary \ref{c:newdir} and Lemma \ref{l:z2z3notzero}(2) shows that for these points we have spanning fields. If $\mathbf{(b_1,b_2)=(0,0)}$ then \(\mathcal{F}_{(0,0,a_3,a_4)}^{4}\) is a non-generic smooth fiber for $\Phi_{4}$ and \[\mathcal{F}_{(0,0,a_3,a_4)}^{4}=\mathcal{F}_{(0,0,a_3,a_4)}^{3}\times \mathbb{C}^3.\] Since we assume that $(a_3,a_4)\neq (0,1)$ ( $(a_3,a_4)\neq (0,0)$ is automatic in this case) we know by Corollary \ref{c:newdir}, Lemma \ref{l:z2z3notzero}(3) and Proposition \ref{p:transitivity} that we have spanning fields. This also shows that Proposition \ref{p:transitivity} holds for these fibers when $L=5$.

We now study the stratum of {\bf singular generic fibers}. Here $\mathbf{(a_3,a_4)=(0,1)}$. Again notice that when $\mathbf{(b_1,b_2)\neq (0,0)}$ then \[\mathcal{F}_{(b_1,b_2,0,1)}^{4}\] is a generic smooth fiber for $\Phi_{4}$ and Proposition \ref{p:transitivity} (for $L=4$), Corollary \ref{c:newdir} and Lemma \ref{l:z2z3notzero} shows that for these points we have spanning fields as above. Next we study the case $\mathbf{(b_1,b_2)=(0,0)}$. In this case we see that \[\mathcal{F}_{(0,0,0,1)}^{4} \cong \mathcal{F}_{(0,0,0,1)}^{3}\times \mathbb{C}^3.\] We write, as in Section \ref{s:proof3factors}, \[ \mathcal{F}_{(0,0,0,1)}^{3}=A_1\cup A_2.\] In $A_2\setminus A_1$ we can use the argument as in the smooth generic case in Section \ref{s:proof3factors}: $z_2z_3\not\equiv 0$ and  $\partial_{z_2w_2w_3}^2(z_2)=z_3^2$ makes it possible to leave the set where $z_2=0$ and
$ \partial_{z_3w_1w_2}^2(z_3)=z_2^2 $ makes it possible to leave the set where $z_3=0$.

Now we need to deal with points in $A_1\times \mathbb{C}^3\subset \mathcal{F}^4_{(0,0,0,1)}$. Because of the inclusion we find $z_5=z_6=0$. Define $C=\{z_2=z_3=z_5=z_6=0\} \subset \mathcal{F}^4_{(0,0,0,1)} \subset \mathcal{F}^5_{(0,0,0,1)}$ which contains the set of singularities  \[ \operatorname{Sing}\left(\mathcal{F}^5_{(0,0,0,1)}\right)=C\cap \left\{\operatorname{Rank} \begin{pmatrix}
w_1 & w_2 & w_4 & w_5 \\ w_2 & w_3 & w_5 & w_6
\end{pmatrix}<2\right\} \] In order to prove Proposition \ref{p:transitivity} and Proposition \ref{p:mainprop} we need to show that fields from $\mathcal{Q}_5$ move out from $C\setminus \operatorname{Sing}\left(\mathcal{F}^5_{(0,0,0,1)}\right)$. Calculating the partial derivatives of $P_1^4,\dots, P_4^4$ in points of \( C \) we find that the ones that are non-zero are those listed in Table 3. \begin{table} 
\renewcommand*{\arraystretch}{1.5}
\begin{tabulary}{0.9\textwidth}{|C|C|C|C|C|}
\hline
 $ \phantom{a}$ & $\partial / \partial z_2 $ & $\partial / \partial z_3 $ & $\partial / \partial z_5$  & $\partial / \partial z_6 $ \\
\hline
$ P_1^4 $ & $1+w_1z_4$ & $w_3z_4$ & 1 & 0 \\
\hline
$ P_2^4 $ & 0 & 1 & 0 & 1 \\
\hline
$ P_3^4 $ & $w_1+w_4+w_1w_4z_4$ & $w_2+w_5+w_2w_4z_4$ & $w_4$ & $w_5$ \\
\hline
$ P_4^4 $ & $w_2+w_5+w_1w_5z_4$ & $w_3+w_6+w_2w_5z_4$ & $w_5$ & $w_6$ \\
\hline

\end{tabulary}
\caption{The non-zero partial derivatives of $P_1^4, P_2^4, P_3^4$, and $P_4^4$.} \label{table3}
\end{table} We examine the complete field $\phi_{z_2z_3z_6}^5$. This field has some complicated components which on $C$ take the form  \[ \phi_{z_2z_3z_6}^5 = \mathcal{D}_1\frac{\partial}{\partial z_2}+\mathcal{D}_2\frac{\partial}{\partial z_3}+\mathcal{D}_3\frac{\partial}{\partial z_6}+\dots \] where \[ \mathcal{D}_1 = \operatorname{det}\begin{pmatrix} w_2+w_5+w_2w_4z_4 & w_5 \\ w_3+w_6+w_2w_5z_4 & w_6 \end{pmatrix}, \] \[ \mathcal{D}_2 = \operatorname{det}\begin{pmatrix} w_1+w_4+w_1w_4z_4 & w_5 \\ w_2+w_5+w_1w_5z_4 & w_6 \end{pmatrix}, \] and \[ \mathcal{D}_3 = \operatorname{det} \begin{pmatrix} w_1+w_4+w_1w_4z_4 & w_2+w_5+w_2w_4z_4 \\ w_2+w_5+w_1w_5z_4 & w_3+w_6+w_2w_5z_4 \end{pmatrix}. \] Whenever at least one of \( \mathcal{D}_1 \), \( \mathcal{D}_2 \) or \( \mathcal{D}_3 \) is non-zero we can move out of \( C \).  Suppose we are in a point of $C\setminus \operatorname{Sing}\left(\mathcal{F}^5_{(0,0,0,1)}\right)$ where \( \mathcal{D}_1=\mathcal{D}_2=\mathcal{D}_3=0 \). Observe that \[ \begin{aligned} \operatorname{Rank} & \begin{pmatrix} w_1 & w_2 & w_4 & w_5 \\ w_2 & w_3 & w_5 & w_6  \end{pmatrix} = \\ & = \operatorname{Rank}  \begin{pmatrix} w_1+w_4+w_1w_4z_4 & w_2+w_5+w_2w_4z_4 & w_4 & w_5 \\ w_2+w_5+w_1w_5z_4 & w_3+w_6+w_2w_5z_4 & w_5 & w_6  \end{pmatrix} \end{aligned} \] (in this case it is 2) since \[ \begin{pmatrix}
 w_1+w_4+w_1w_4z_4 \\ w_2+w_5+w_1w_5z_4 \end{pmatrix}=\begin{pmatrix}
 w_1 \\ w_2
 \end{pmatrix} + (1+w_1z_4) \begin{pmatrix}
 w_4 \\ w_5
 \end{pmatrix}
  \] and \[ \begin{pmatrix}
 w_2+w_5+w_2w_4z_4 \\ w_3+w_6+w_2w_5z_4 \end{pmatrix}=\begin{pmatrix}
 w_2 \\ w_3
 \end{pmatrix} + \begin{pmatrix}
 w_5 \\ w_6
 \end{pmatrix}  + w_2z_4 \begin{pmatrix}
 w_4 \\ w_5
 \end{pmatrix}.
  \] The fact that \( \mathcal{D}_1=\mathcal{D}_2=\mathcal{D}_3=0 \) means that the rank drops when we remove the third column from these matrices. This implies that the third column is non-zero and the other columns are multiples of a non-zero vector \( v \) which moreover is linearly independent of the third column. Now we use the field \( \gamma^3 \) (see (\ref{e:gammaodd})) to show that the set  \[ I=C\setminus \operatorname{Sing}(\mathcal{F}^5_{(0,0,0,1)})\cap \{\mathcal{D}_1=\mathcal{D}_2=\mathcal{D}_3=0\}  \]  does not contain an invariant subset under fields from \( \mathcal{Q}_5 \). In the points that we are considering $\gamma^3=\frac{\partial}{\partial z_4}$. We consider two cases. 
  {\bf Case 1: $(w_5,w_6)\neq (0,0)$} In this case \[ \det \begin{pmatrix}
  w_4 & w_5 \\ w_5 & w_6
  \end{pmatrix}\neq 0.  \] We have \[ \gamma^3(\mathcal{D}_1)=w_2\det \begin{pmatrix}
  w_4 & w_5 \\ w_5 & w_6
  \end{pmatrix} \] Thus \( \gamma^3 \) moves points out of \( I \) unless \( w_2=0 \). Looking at \[ \gamma^3(\mathcal{D}_2)=w_1\det \begin{pmatrix}
  w_4 & w_5 \\ w_5 & w_6
  \end{pmatrix} \] we see that $w_1=0$ for $I$ to be invariant. Assuming in addition $w_1=w_2=0$ we find that \[ \mathcal{D}_2=\det \begin{pmatrix}
  w_4 & w_5 \\ w_5 & w_6
  \end{pmatrix} \] which is a contradiction since $\mathcal{D}_2=0$ on $I$. {\bf Case 2: $(w_5,w_6) = (0,0)$}. This implies \( w_4\neq 0 \). On these assumptions \[ \mathcal{D}_3=(w_1w_3-w_2^2)(1+z_4w_4)+w_3w_4 \] and \[ \gamma^3(\mathcal{D}_3)=(w_1w_3-w_2^2)w_4. \] Now \( \gamma^3(\mathcal{D}_3)=0 \) implies that $w_1w_3-w_2^2=0$ which in combination with  \( \mathcal{D}_3=0 \) implies that $w_3=0$. This in turn gives $w_2=0$ and \[ \begin{pmatrix}
    w_1 & w_2 & w_{4} & w_{5} \\ w_2 & w_3 & w_{5} & w_{6}
    \end{pmatrix} = \begin{pmatrix}
    w_1 & 0 & w_{4} & 0 \\ 0 & 0 & 0 & 0
    \end{pmatrix} \] which contradicts the assumption that \[ \operatorname{Rank}\begin{pmatrix}
    w_1 & w_2 & w_{4} & w_{5} \\ w_2 & w_3 & w_{5} & w_{6}
    \end{pmatrix} = 2. \]
 
Finally we study the stratum of {\bf non-generic fibers}, that is $\mathbf{(a_3,a_4)=(0,0)}$. Here all fibers are smooth. Also \[\mathcal{F}_{(a_1,a_2,0,0)}^{5}=\mathcal{F}_{(a_1,a_2,0,0)}^{4}\times \mathbb{C}^3\] and since $\mathcal{F}_{(a_1,a_2,0,0)}^{4}$ is elliptic we are done.

\section{Proof of Proposition \ref{p:mainprop}: Induction steps}\label{s:induction}
Recall the description of the stratification for the submersion $\Phi_M = \pi_4 \circ \Psi_M$ given in Section \ref{s:stratification}. When $M$ is {\bf odd} we have the following strata:
\begin{itemize}  
\item The strata of {\bf generic fibers}: When $ (a_3,a_4) \neq (0,0) $. The fibers are graphs over $ \mathcal{G}^{M-1}_{(a_3,a_4)} \times \mathbb{C} $. This set is divided into two strata as follows:
\begin{itemize}
\item   Smooth generic fibers:  When $ (a_3,a_4) \neq (0,1) $ then the fibers are smooth.
\item  Singular generic fibers: When $ (a_3,a_4) = (0,1) $ then the fibers are non-smooth.
\end{itemize} 
\item The stratum of {\bf non-generic fibers}: When $ (a_3,a_4) = (0,0) $ the fibers are $\mathcal{F}^M_{(a_1,a_2,0,0)}=\mathcal{F}^{M-1}_{(a_1,a_2,0,0)} \times \mathbb{C}^3  $. Moreover the fibers are smooth.
 \end{itemize} 

When $M$ is {\bf even} we have the following strata:
 \begin{itemize}  
\item The stratum of {\bf generic fibers}: When $ (a_1,a_2) \neq (0,0) $. The fibers are graphs over $ \mathcal{H}^{M-1}_{(a_1,a_2)} \times \mathbb{C} $. Moreover the fibers are smooth. 
\item The strata of {\bf non-generic fibers}: When $ (a_1,a_2) = (0,0) $ the fibers are $\mathcal{F}^M_{(0,0,a_3,a_4)}=\mathcal{F}^{M-1}_{(0,0,a_3,a_4)} \times \mathbb{C}^3  $. This set is divided into two strata as follows:
\begin{itemize}
\item   Smooth non-generic fibers:  When $ (a_3,a_4) \neq (0,1) $ then the fibers are smooth.
\item  Singular non-generic fibers: When $ (a_3,a_4) = (0,1) $ then the fibers are non-smooth.
\end{itemize} 
 \end{itemize}

We will now complete the proof by doing the induction steps necessary.

\subsection{Even number of factors}\label{ss:evenK}
We begin by showing that the stratified submersion is elliptic when the number of matrix factors is even. This case is easier than the case when the number of factors is odd which we will deal with in subsection \ref{ss:oddK}. Assume that $K=2k-1\ge 5$ and that the submersions $\Phi_L=\pi_4 \circ \Psi_L$ are stratified elliptic submersions when $3\le L \le K$ and that Propostion \ref{p:transitivity} is true when $3\le L \le K$. 

Study \begin{equation}\label{e:unioneven}
\mathcal{F}_{(a_1,a_2,a_3,a_4)}^{K+1}=\mathcal{F}_{(a_1,a_2,a_3,a_4)}^{2k}=\bigcup_{(w_{3k-2},w_{3k-1},w_{3k})\in \mathbb{C}^3} \mathcal{F}_{(a_1,a_2,b_3,b_4)}^{2k-1} 
\end{equation} 
where $b_3=a_3-w_{3k-2}a_1-w_{3k-1}a_2$ and $b_4=a_4-w_{3k-1}a_1-w_{3k}a_2$. That is we use the new group of variables $w_{3k-2},w_{3k-1}$ and $w_{3k}$ to present $\mathcal{F}_{(a_1,a_2,a_3,a_4)}^{2k}$ as a fibration over $\mathbb{C}^3$ with fibers $\mathcal{F}^{2k-1}$. 

Let us describe the strategy similar to the case of 4 and 5 matrix factors. We like to use Proposition \ref{p:transitivity} for $K=2k-1$ which gives us complete fields that span along that fibration. Next we want to find complete fields among those that are tangential to $\mathcal{F}_{(a_1,a_2,a_3,a_4)}^{2k}$ that also are transversal to the fibers in the fibration. We will appeal to Corollary \ref{c:newdir} and Lemma \ref{l:z2z3notzero}(1) to find these fields. Taken together this will show that a subset $A$ in the fiber $\mathcal{F}_{(a_1,a_2,a_3,a_4)}^{2k}$ that is invariant with respect to vector fields from $\mathcal{Q}_{2k}$ must be contained in the union of non-generic fibers $\mathcal{F}^{2k-1}$ and singular points of generic fibers $\mathcal{F}^{2k-1}$. Call this union $\mathcal{U}_{(a_1,a_2,a_3,a_4)}^{2k}$. Our aim will then be to show that there cannot exist such an invariant set $A$ by showing that every point in $\mathcal{U}_{(a_1,a_2,a_3,a_4)}^{2k}$ can be moved into $\mathcal{F}_{(a_1,a_2,a_3,a_4)}^{2k}\setminus \mathcal{U}_{(a_1,a_2,a_3,a_4)}^{2k}$ by vector fields in $\mathcal{Q}_{2k}$. 

We now take care of the details. Because of (\ref{e:unioneven}) there are $b_3$ and $b_4$ so that $\vec{Z}_{2k-1}\in \mathcal{F}_{(a_1,a_2,b_3,b_4)}^{2k-1}$ and \[\vec{Z}_{2k}=(\vec{Z}_{2k-1},w_{3k-2},w_{3k-1},w_{3k}).\]  We begin by studying the stratum of {\bf generic fibers}, that is $\mathbf{(a_1,a_2)\neq (0,0)}$. 
For points where $\mathbf{(b_3,b_4) \notin \{(0,0),(0,1)\}}$ then $\mathcal{F}_{(a_1,a_2,b_3,b_4)}^{2k-1}$ is a smooth generic fiber for the submersion $\Phi_{2k-1}$ and Propostion \ref{p:transitivity} (for $L=2k-1$) together with Corollary \ref{c:newdir} and Lemma \ref{l:z2z3notzero}(1) let us conclude that we have complete vector fields spanning the tangent space of $\mathcal{F}_{(a_1,a_2,a_3,a_4)}^{2k}$ at these points. For points where $\mathbf{(b_3,b_4)=(0,0)}$ we have \[\mathcal{F}^{2k-1}_{(a_1,a_2,0,0)}=\mathcal{F}^{2k-2}_{(a_1,a_2,0,0)} \times \mathbb{C}^3\] and Proposition \ref{p:transitivity} (for $L=2k-2$ applied to the first factor) together with Corollary \ref{c:newdir} and Lemma \ref{l:z2z3notzero}(1) (Lemma \ref{l:z2z3notzero}(2) when $2k-2=4$) shows that we have spanning fields in these points. When $\mathbf{(b_3,b_4)=(0,1)}$ then $\mathcal{F}^{2k-1}_{(a_1,a_2,0,1)}$ is a singular generic fiber for $\Phi_{2k-1}$ and at smooth points of the fiber we have complete spanning fields
by Proposition \ref{p:transitivity} (for $L=2k-1$), Corollary \ref{c:newdir} and Lemma \ref{l:z2z3notzero}. It remains to study \[\vec{Z}_{2k-1}\in \operatorname{Sing}\left(\mathcal{F}^{2k-1}_{(a_1,a_2,0,1)}\right)\] which is given by \begin{equation} \label{e:condonZs}
    z_2=z_3=z_5=z_6=\dots =z_{3k-4}=z_{3k-3}=0
\end{equation} and \begin{equation}\label{e:rank}
    \operatorname{Rank}\begin{pmatrix}
    w_1 & w_2 & \dots & w_{3k-5} & w_{3k-4} \\ w_2 & w_3 & \dots & w_{3k-4} & w_{3k-3}
    \end{pmatrix} < 2.
\end{equation} A calculation assuming (\ref{e:condonZs}) shows that \[ \begin{aligned} \partial_{z_{3k-4}z_{3k-3}z_{3k}}^{2k-1}&=(z_{3k-2}w_{3k-4}+z_{3k-1}w_{3k-3})\frac{\partial}{\partial z_{3k-4}}-\\&-(1+z_{3k-2}w_{3k-5}+z_{3k-1}w_{3k-4})\frac{\partial}{\partial z_{3k-3}}+\dots \end{aligned} \] Therefore the complete fields \[\theta^{2k}_{z_{3k-4}z_{3k-3}z_{3k}}=a_1^2\partial^{2k-1}_{z_{3k-4}z_{3k-3}z_{3k}}+\dots\] and \[\phi^{2k}_{z_{3k-4}z_{3k-3}z_{3k}}=a_2^2\partial^{2k-1}_{z_{3k-4}z_{3k-3}z_{3k}}+\dots\] move points out of $\operatorname{Sing}\left( \mathcal{F}^{2k-1}_{(a_!,a_2,0,1)}\right)$ (into the big orbit) unless in addition to
(\ref{e:condonZs}) and (\ref{e:rank}) also \begin{equation}\label{e:extracondeven}
 	    z_{3k-2}w_{3k-4}+z_{3k-1}w_{3k-3}=1+z_{3k-2}w_{3k-5}+z_{3k-1}w_{3k-4}=0.
 	\end{equation} Points in an invariant subset must satisfy also these equations. A calculation assuming (\ref{e:condonZs}) gives that $\partial^{2k-1}_{z_{3k-2}z_{3k-1}z_{3k}}=\frac{\partial}{\partial z_{3k-2}}$. Therefore the complete fields \[\theta^{2k}_{z_{3k-2}z_{3k-1}z_{3k}}=a_1^2\partial^{2k-1}_{z_{3k-2}z_{3k-1}z_{3k}}+\dots\]
 	and \[\theta^{2k}_{z_{3k-2}z_{3k-1}z_{3k}}=a_2^2\partial^{2k-1}_{z_{3k-2}z_{3k-1}z_{3k}}+\dots\] move points out of this set since 
 	
 	\[\theta^{2k}_{z_{3k-2}z_{3k-1}z_{3k}}(1+z_{3k-2}w_{3k-5}+z_{3k-1}w_{3k-4})=a_1^2w_{3k-5}, \]
 	\[\theta^{2k}_{z_{3k-2}z_{3k-1}z_{3k}}(z_{3k-2}w_{3k-4}+z_{3k-1}w_{3k-3})=a_1^2w_{3k-4}, \]
 	\[\phi^{2k}_{z_{3k-2}z_{3k-1}z_{3k}}(1+z_{3k-2}w_{3k-5}+z_{3k-1}w_{3k-4})=a_2^2w_{3k-5}, \]
 	 \[\phi^{2k}_{z_{3k-2}z_{3k-1}z_{3k}}(z_{3k-2}w_{3k-4}+z_{3k-1}w_{3k-3})=a_2^2w_{3k-4}, \]
 	cannot all be zero, because this would contradict (\ref{e:extracondeven}). 
 Notice that this proves Proposition \ref{p:transitivity} for $L=2k$. 

Now we study the stratum of {\bf non-generic fibers}, that is $\mathbf{a_1=a_2=0}$. In this case we know that \[\mathcal{F}^{2k}_{(0,0,a_3,a_4)}= \mathcal{F}^{2k-1}_{(0,0,a_3,a_4)}\times \mathbb{C}^3\] and by the induction assumption we are done. This finishes the induction step for an even number of factors.


\subsection{Odd number of factors}\label{ss:oddK}

We assume that $K=2k\ge 6$ and that the submersions $\Phi_L=\pi_4 \circ \Psi_L$ are stratified elliptic submersions when $3\le L \le K$ and that Propostion \ref{p:transitivity} is true when $3\le L \le K$. 

Study \begin{equation}\label{e:unionodd}
\mathcal{F}_{(a_1,a_2,a_3,a_4)}^{K+1}=\mathcal{F}_{(a_1,a_2,a_3,a_4)}^{2k+1}=\bigcup_{(z_{3k+1},z_{3k+2},z_{3k+3})\in \mathbb{C}^3} \mathcal{F}_{(b_1,b_2,a_3,a_4)}^{2k} 
\end{equation} 
where $b_1=a_1-z_{3k+1}a_3-z_{3k+2}a_4$ and $b_2=a_2-z_{3k+2}a_3-z_{3k+3}a_4$. Let $\vec{Z}_{2k+1}\in \mathcal{F}_{(a_1,a_2,a_3,a_4)}^{2k+1}$. Because of (\ref{e:unionodd}) there are $b_1$ and $b_2$ so that $\vec{Z}_{2k}\in \mathcal{F}_{(b_1,b_2,a_3,a_4)}^{2k}$ and \[\vec{Z}_{2k+1}=(\vec{Z}_{2k},z_{3k+1},z_{3k+2},z_{3k+3}).\] 

Begin with the stratum of {\bf smooth generic fibers}, that is $$\mathbf{(a_3,a_4)\not\in  \{(0,0),(0,1)\} }$$. First notice that if $\mathbf{(b_1,b_2)\neq (0,0)}$ then \(\mathcal{F}_{(b_1,b_2,a_3,a_4)}^{2k}\) is a generic smooth fiber for $\Phi_{2k}$ and as above Proposition \ref{p:transitivity} (for $L=2k$), Corollary \ref{c:newdir} and Lemma \ref{l:z2z3notzero} shows that for these points we have spanning fields. If $\mathbf{(b_1,b_2)=(0,0)}$ then \[\mathcal{F}_{(0,0,a_3,a_4)}^{2k} \cong \mathbb{C}^3\times \mathcal{F}_{(0,0,a_3,a_4)}^{2k-1} \] is a non-generic smooth fiber for $\Phi_{2k}$ and since \( \mathcal{F}_{(0,0,a_3,a_4)}^{2k-1} \) is a generic smooth fiber Proposition \ref{p:transitivity} (for $L=2k-1$), Corollary \ref{c:newdir} and Lemma \ref{l:z2z3notzero} shows that for these points we have spanning fields.

We now study the stratum of {\bf singular generic fibers}. Here $\mathbf{(a_3,a_4)=(0,1)}$. Again notice that when $\mathbf{(b_1,b_2)\neq (0,0)}$ then \(\mathcal{F}_{(b_1,b_2,0,1)}^{2k}\) is a generic smooth fiber for $\Phi_{2k}$ and Proposition \ref{p:transitivity} (for $L=2k$), Corollary \ref{c:newdir} and Lemma \ref{l:z2z3notzero} shows that for these points we have spanning fields as above. Next we study the case $\mathbf{(b_1,b_2)=(0,0)}$. In this case we see that \(\mathcal{F}_{(0,0,0,1)}^{2k}\) is a singular non-generic fiber of $\Phi_{2k}$ and \[ \mathcal{F}_{(0,0,0,1)}^{2k}\cong \mathcal{F}_{(0,0,0,1)}^{2k-1}\times \mathbb{C}^3_{w_{3k-2}w_{3k-1}w_{3k}} .\] The smooth points of \( \mathcal{F}_{(0,0,0,1)}^{2k-1} \) (which is generic) are handled using Proposition \ref{p:transitivity} (for $L=2k-1$), Corollary \ref{c:newdir} and Lemma \ref{l:z2z3notzero}. We have the following chain of inclusions \begin{equation*}
    \begin{aligned}
    \mathcal{F}_{(a_1,a_2,0,1)}^{2k-1}\supset \mathcal{F}_{(0,0,0,1)}^{2k} = \\ = \mathcal{F}_{(0,0,0,1)}^{2k-1}\times \mathbb{C}^3 \supset \operatorname{Sing}(\mathcal{F}_{(0,0,0,1)}^{2k})\times \mathbb{C}^3 \supset \\ \supset \operatorname{Sing}(\mathcal{F}_{(a_1,a_2,0,1)}^{2k+1})  
    \end{aligned}
\end{equation*} By the arguments above any possible invariant subset must be contained in  \[ J= (\operatorname{Sing}(\mathcal{F}_{(0,0,0,1)}^{2k})\times \mathbb{C}^3) \setminus  \operatorname{Sing}(\mathcal{F}_{(a_1,a_2,0,1)}^{2k+1}) \] Points in \( J \) are characterized by \( z_2=z_3= \dots = z_{3k-4}=z_{3k-3}=z_{3k-1}=z_{3k}=0  \),  \[ \operatorname{Rank}\begin{pmatrix}
w_1 & w_2 & \dots & w_{3k-5} & w_{3k-4} \\ w_2 & w_3 & \dots & w_{3k-4} & w_{3k-3}
\end{pmatrix} < 2 \] and \[ \operatorname{Rank}\begin{pmatrix}
w_1 & w_2 & \dots & w_{3k-5} & w_{3k-4} & w_{3k-2} & w_{3k-1} \\ w_2 & w_3 & \dots & w_{3k-4} & w_{3k-3} & w_{3k-1} & w_{3k}
\end{pmatrix} = 2. \] Take the largest \( l<k \) such that \[ \operatorname{Rank}\begin{pmatrix}
 w_{3l-2} & w_{3l-1} \\  w_{3l-1} & w_{3l}
\end{pmatrix} =1. \] Let \( \hat{Z}= \sum_{j=l+1}^k z_{3j-2} \). We examine the complete field $\phi_{z_{3l-1}z_{3l}z_{3k}}^{2k+1}$.  This field has some complicated components which on $J$ take the form  \[ \phi_{z_{3l-1}z_{3l}z_{3k}}^{2k+1} = \mathcal{D}_1\frac{\partial}{\partial z_{3l-1}}+\mathcal{D}_2\frac{\partial}{\partial z_{3l}}+\mathcal{D}_3\frac{\partial}{\partial z_{3k}}+\dots \] where \[ \mathcal{D}_1 = \operatorname{det} \begin{pmatrix} w_{3l-1}+w_{3k-1}+w_{3l-1}w_{3k-2}\hat{Z} & w_{3k-1} \\ w_{3l}+w_{3k}+w_{3l-1}w_{3k-1}\hat{Z} & w_{3k} \end{pmatrix}, \] \[ \mathcal{D}_2 = \operatorname{det} \begin{pmatrix} w_{3l-2}+w_{3k-2}+w_{3l-2}w_{3k-2}\hat{Z} & w_{3k-1} \\ w_{3l-1}+w_{3k-1}+w_{3l-2}w_{3k-1}\hat{Z} & w_{3k} \end{pmatrix}, \] and

\[ \mathcal{D}_3 = \operatorname{det}\begin{pmatrix} w_{3l-2}+w_{3k-2}+w_{3l-2}w_{3k-2}\hat{Z} & w_{3l-1}+w_{3k-1}+w_{3l-1}w_{3k-2}\hat{Z} \\ w_{3l-1}+w_{3k-1}+w_{3l-2}w_{3k-1}\hat{Z} & w_{3l}+w_{3k}+w_{3l-1}w_{3k-1}\hat{Z} \end{pmatrix}. \]

Whenever at least one of \( \mathcal{D}_1 \), \( \mathcal{D}_2 \) or \( \mathcal{D}_3 \) is non-zero we can move out of \( J \).  Now suppose we are in a point of $J$ where \( \mathcal{D}_1=\mathcal{D}_2=\mathcal{D}_3=0 \).

  Let \begin{equation*} \begin{aligned} &\mathcal{C}=\\&=
   \begin{pmatrix} w_{3l-2}+w_{3k-2}+w_{3l-2}w_{3k-2}\hat{Z} & w_{3l-1}+w_{3k-1}+w_{3l-1}w_{3k-2}\hat{Z} & w_{3k-2} & w_{3k-1} \\ w_{3l-1}+w_{3k-1}+w_{3l-2}w_{3k-1}\hat{Z} & w_{3l}+w_{3k}+w_{3l-1}w_{3k-1}\hat{Z} & w_{3k-1} & w_{3k}  \end{pmatrix}\end{aligned}
  \end{equation*} and observe that \begin{equation*} 
2=\operatorname{Rank} \begin{pmatrix} w_{3l-2} & w_{3l-1} & w_{3k-2} & w_{3k-1} \\ w_{3l-1} & w_{3l} & w_{3k-1} & w_{3k}  \end{pmatrix} = \operatorname{Rank} \mathcal{C}     
   \end{equation*}   by column operations. 
   
The fact that \( \mathcal{D}_1=\mathcal{D}_2=\mathcal{D}_3=0 \) means that the rank drops when we remove the third column from these matrices. This implies that the third column is non-zero and the other columns are multiples of a non-zero vector \( v \) which moreover is linearly independent of the third column. Now we use the field \( \gamma^{3l} \) (see (\ref{e:gammaodd}) or (\ref{e:gammaeven})) to show that the set  \[ I=J\cap \{\mathcal{D}_1=\mathcal{D}_2=\mathcal{D}_3=0\}  \]  does not contain an invariant subset under fields from \( \mathcal{Q}_{2k+1} \). In the points that we are considering $\gamma^{3l}=\frac{\partial}{\partial z_{3l+1}}$. We consider two cases. 
  {\bf Case 1: $(w_{3k-1},w_{3k})\neq (0,0)$} In this case \[ \det \begin{pmatrix}
  w_{3k-2} & w_{3k-1} \\ w_{3k-1} & w_{3k}
  \end{pmatrix}\neq 0.  \] 
  
  We have \[ \gamma^{3l}(\mathcal{D}_1)=w_{3l-1}\det \begin{pmatrix}
  w_{3k-2} & w_{3k-1} \\ w_{3k-1} & w_{3k}
  \end{pmatrix} \] Thus \( \gamma^{3l} \) moves points out of \( I \) unless \( w_{3l-1}=0 \). Looking at \[ \gamma^{3l}(\mathcal{D}_2)=w_{3l-2}\det \begin{pmatrix}
  w_{3k-2} & w_{3k-1} \\ w_{3k-1} & w_{3k}
  \end{pmatrix} \] we see that $w_{3l-2}=0$ for $I$ to be invariant. Assuming in addition $w_{3l-2}=w_{3l-1}=0$ we find that \[ \mathcal{D}_2=\det \begin{pmatrix}
  w_{3k-2} & w_{3k-1} \\ w_{3k-1} & w_{3k} 
  \end{pmatrix} =0 \] which is a contradiction. {\bf Case 2: $(w_{3k-1},w_{3k}) = (0,0)$}. This implies \( w_{3k-2}\neq 0 \). On these assumptions \[ \mathcal{D}_3=(w_{3l-2}w_{3l}-w_{3l-1}^2)(1+w_{3k-2}\hat{Z})+w_{3l}w_{3k-2} \] and \[ \gamma^{3l}(\mathcal{D}_3)=(w_{3l-2}w_{3l}-w_{3l-1}^2)w_{3k-2}. \] Now \( \mathcal{D}_3=\gamma^{3l}(\mathcal{D}_3)=0 \) implies that $(w_{3l-2}w_{3l}-w_{3l-1}^2)=0$ and $w_{3l}=0$. The first equality  gives $w_{3l-1}=0$ which alltogether contradicts the assumption that \[ \operatorname{Rank}  \begin{pmatrix} w_{3l-2} & w_{3l-1} & w_{3k-2} & w_{3k-1} \\ w_{3l-1} & w_{3l} & w_{3k-1} & w_{3k}  \end{pmatrix}  = 2. \] 
 
Finally we study the stratum of {\bf non-generic fibers}, that is $\mathbf{(a_3,a_4)=(0,0)}$. Here all fibers are smooth. Also \[\mathcal{F}_{(a_1,a_2,0,0)}^{2k+1}=\mathcal{F}_{(a_1,a_2,0,0)}^{2k}\times \mathbb{C}^3\] and since $\mathcal{F}_{(a_1,a_2,0,0)}^{2k}$ is elliptic by the induction hypothesis we are done.  
\section{Product of exponentials and open questions}\label{s:exponential}

For a Stein space $X$, a complex Lie group $G$ and its exponential map $\exp: \mathfrak{g} \to G$ we say that a holomorphic map $f:X \to G$ 
is a product of $k$ exponentials if there are holomorphic maps $f_1, \ldots, f_k:X \to \mathfrak{g}$ such that $$f=\exp(f_1)\cdots \exp(f_k).$$
It is easy to see that any map $f$ which is a product of exponentials (for some sufficiently large $k$) is null-homotopic. 
In the case where $G$ is the special linear group $\operatorname{SL}_n(\mathbb{C})$ the converse follows from~\cite{Ivarsson:2012} as explained in~\cite{DK}. In the same way we prove:

\begin{Thm} \label{exponential}  For a Stein space $X$ there is
a number $N$ depending on the dimension of
$X$ such that any null-homotopic holomorphic map  $f\colon X\to \operatorname{Sp}_{4}(\mathbb{C})$ can be factorized as 
\[f(x)= \exp (G_1 (x))\dots \exp (G_{K}(x)).\]

where $G_i : X \to \mathfrak{sp_4} (\mathbb{C})$ are holomorphic maps.

\end{Thm}
\begin{proof}
By Theorem \ref{t:mainthmrestate} we find
$K$ elementary symplectic matrices $A_i (x)
\in \operatorname{Sp}_4 (\mathcal{O}(X))$, 
$ i = 1, 2, \ldots K$, such that 
\[f(x)=  A_1 (x)\dots A_{K}(x).\]

Now remark that the logarithmic series 
$$\ln (Id + B) = \sum \frac{1}{n} B^n $$ is finite for the nilpotent matrices $B_i = A_i -Id$.
\end{proof}

\begin{OP} Determine the 
optimal number $K$ in Theorem \ref{exponential}.
\end{OP}

\begin{OP}
 Determine the optimal numbers of factors in Theorem \ref{t:mainthmrestate}.
\end{OP}

The smooth fibers \[\mathcal{F}^K_{(a_1,a_2,a_3,a_4)} = (\pi_4 \circ \Psi_K)^{-1}(a_1,a_2,a_3,a_4). \] of the fibration projecting the product of $K$ elementary symplectic matrices to its last row are smooth affine algebraic varieties. They are new examples of Oka manifolds, since  we prove as a by-product of Proposition \ref{p:mainprop} that they are
holomorphically flexible (for definition see \cite{AFKKZ}). 
Our proof does not give the algebraic flexibility of them. Even if our
initial complete fields obtained in Section \ref{s:descComplete} are algebraic. The problem is that their flows are not always algebraic (not all of them are locally nilpotent). Therefore the pull-backs by their flows are merely holomorphic vector fields.

\begin{OP}

Which other (stronger) flexibility properties 
like algebraic flexibility, algebraic (volume) density property,
or (volume) density property do the fibers $\mathcal{F}^K_{(a_1,a_2,a_3,a_4)}$ admit?
\end{OP}

For the definition of these flexibility properties we refer to the 
overview article \cite{Kutzschebauch}.

Let us remark that the fibers of the fibration for $5$ elementary factors in \cite{Ivarsson:2012} have been thoroughly studied in \cite{KaKu}
and \cite{KaKu1} Section 7. They were the starting point for the introduction of
the class of generalized Gizatullin surfaces whose final classification
was achieved in \cite{KKL}.  The topology of these fibers for any number of elementary factors has been studied in \cite{DVi}  where
it was also proven that they admit the algebraic volume density property.
Such studies are interesting since the possible topological types of Oka manifolds or manifolds with the density property are not understood at the moment.

\begin{OP}

Determine the homology groups of the fibers $\mathcal{F}^K_{(a_1,a_2,a_3,a_4)}$.
\end{OP}
And finally: 

\begin{OP}
Prove Conjecture \ref{c:factorization}.
\end{OP}

\end{document}